\documentclass[final,authoryear,times,12pt]{elsarticle}
\usepackage{amsfonts}
\usepackage{url}
\usepackage[colorlinks,citecolor=blue,urlcolor=blue,linkcolor=blue]{hyperref}
\usepackage{bm}
\usepackage{amsmath}
\usepackage{amsthm}
\usepackage{amssymb}
\usepackage{pgf,tikz,pgfplots}
\pgfplotsset{compat=1.14}
\usepackage{mathrsfs}
\usetikzlibrary{arrows}
\usepackage{xcolor}
\usepackage{caption}
\usepackage{eurosym}
\usepackage{type1cm}
\usepackage{makeidx}
\usepackage{multicol}
\usepackage[bottom]{footmisc}
\usepackage{natbib}

\numberwithin{equation}{section}
\newlength{\defbaselineskip}
\newtheorem{thm}{Theorem}[section]
\newtheorem{coro}[thm]{Corollary}
\newtheorem{lemma}[thm]{Lemma}
\newtheorem{prop}[thm]{Proposition}
\newtheorem{ass}[thm]{Assumption}

\theoremstyle{remark}

\theoremstyle{remark}
\newtheorem{example}[thm]{Example}
\theoremstyle{definition}

\newcommand{\pof}{prediction}
\newcommand{\VaR}{Q}
\newcommand{\ES}{ES}
\newcommand{\myb}{\underline{\beta}}

\newcommand{\myG}{G}
\newcommand{\mycol}{black}

\newcommand{\PP}[1]{\textcolor{\mycol}{#1}}

\input{ee.sty}

\begin{document}

\begin{frontmatter}




\title{\vspace{-4cm}GARCH density and functional forecasts\tnoteref{t1}}


\author[KMA]{Karim M. Abadir\fnref{a1}}
\author[AL]{Alessandra Luati\fnref{a2}}
\author[JRC]{Paolo Paruolo\fnref{a3}}

\address[MKA]{Business School, Imperial College London, UK}
\address[AL]{Department of Statistical Sciences ``Paolo Fortunati'', University of Bologna, Italy}
\address[JRC]{European Commission, Joint Research Centre (JRC), Ispra (VA), Italy}

\begin{abstract}
This paper derives the analytic form of the $h$-step ahead \pof\ density of a GARCH(1,1) process under Gaussian innovations, with a possibly asymmetric news impact curve. The contributions of the paper consists both in the derivation of the analytic form of the density, and in its application to a number of econometric problems. A first application of the explicit formulae is to characterize the degree of non-Gaussianity of the \pof\ distribution; for some values encountered in applications, deviations of the \pof\ distribution from the Gaussian are found to be small, and sometimes not. the Gaussian density as an approximation of the true \pof\ density. A second application of the formulae is  to compute exact tail probabilities and functionals, such as the Value at Risk and the Expected Shortfall, that measure risk
when the underlying asset return is generated by a Gaussian GARCH(1,1). This improves on existing methods based on Monte Carlo simulations and (non-parametric) estimation techniques, because the present exact formulae are free of Monte Carlo estimation uncertainty. A third application is the definition of uncertainty regions for functionals of the \pof\ distribution that reflect in-sample estimation uncertainty. These applications are illustrated on selected empirical examples.
\end{abstract}


\tnotetext[t1]{Information and views set out in this paper are those of the authors and do not necessarily reflect the ones of the institutions of affiliation.
The authors acknowledge useful comments from Enrique Sentana, Robert Engle, Eric Renault, Leopoldo Catania, Barbara Rossi, Dimitris Politis, Ngai Chan, Christian Brownlees,
Nour Meddahi,
Torben Andersen,
two anonymous referees,
as well as from participants to seminars in University of Verona and Bocconi University of Milan,
the Seventh Italian Congress of Econometrics and Empirical Economics 2017,
EC$^2$ 2017 in Amsterdam,
Barcellona GSE Summer Forum 2017
and the NBER-NSF Time Series Conference 2018.
}

\fntext[a1]{Email: k.m.abadir@imperial.ac.uk, \href{https://orcid.org/0000-0001-5637-9513}{ORCID: 0000-0001-5637-9513}.}
\fntext[a2]{Email: alessandra.luati@unibo.it, \href{https://orcid.org/0000-0001-6407-9385}{ORCID: 0000-0001-6407-9385}.}
\fntext[a3]{Email: paolo.paruolo@ec.europa.eu, \href{https://orcid.org/0000-0002-3982-4889}{ORCID: 0000-0002-3982-4889}, \\ \indent \indent \textbf{Corresponding author.} \\ \indent \indent Address: European Commission, Joint Research Centre (JRC), Via Enrico Fermi 2749,\\ \indent \indent TP 723, 21027 Ispra (VA), Italy}

\end{frontmatter}

\newpage


\newpage
\tableofcontents
\newpage

\section{Introduction}\label{sec_intro}
Since their introduction in \cite{Engle1982} and \cite{Bollerslev1986},
Generalised AutoRegressive Conditional Heteroskedasticity (GARCH) processes
have been widely employed in financial econometrics, see e.g. \cite{Bollerslev2010b}. In the GARCH original formulation, the conditional distribution of innovations was typically assumed to be Gaussian; even with
Gaussian innovations, GARCH processes were shown to generate 
volatility clustering and, when stationary, an unconditional distribution with fatter tails than the Gaussian, see e.g. \cite{Bollerslev1992}.

GARCH processes can include several lags $q$ of the past squared shocks and
several lags $p$ of the past volatility; in practice, however, the GARCH(1,1) model
with $p=q=1$ is often found to offer a good fit for asset returns, and it is
usually preferred to GARCH models with more parameters, see \cite{Tsay2010}
section 3.5, or \cite{Andersen-etal-2006}, section 3.6. Moreover, many
multivariate GARCH models are built on the univariate GARCH(1,1), see e.g.
\cite{Engle-etal2017} and references therein. In this sense the GARCH(1,1)
is both the prototype and the workhorse of GARCH processes in practice.

GARCH processes map news into the conditional volatility; the
function obtained by replacing past conditional volatilities with
unconditional ones was called by \cite{Engle1993a} the news-impact-curve.
For GARCH(1,1) processes, this curve yields the same value of
volatility for positive and negative shocks, i.e. it is symmetric. \cite%
{Glosten1993} (henceforth GJR) extended the GARCH setup to allow for
asymmetric news impact curve responses to negative shocks.

Many measures of risk are functions of the \pof\ distribution of asset
returns. These measures include the Value at Risk, which is a quantile of
the \pof\ distribution of the asset return, see \cite{Jorion2006}, as
well as the Expected Shortfall, see \cite{Patton2019} and \cite{Arvanitis2018}. The latter is
the expected value of the \pof\ distribution of the asset return in the
left tail of the \pof\ density below the Value at Risk; this measure has
been recently re-emphasised by the Third Basel Accords. Both measures are
functionals of the \pof\ distribution of asset returns. 

The \pof\ distribution of a GARCH(1,1) hence plays an important role for
the computation of risk measures in financial applications. This
distribution is not known in analytic form beyond the
distribution of innovations \PP{for the 1-step ahead case}, which is given by assumption \PP{when building} the process, see e.g. \cite{Andersen-etal-2006}, page 811 and \cite{Baillie1992}.

The present paper derives the analytical form of the $h$-step-ahead \pof\
density of a Gaussian GARCH(1,1), $x_t=\sigma _{t}\varepsilon _{t}$, with $\sigma _{t}^{2}=\omega +\alpha
x_{t-1}^{2}+\beta \sigma _{t-1}^{2}$,
also allowing for GJR-GARCH(1,1) \PP{with} asymmetric news-impact-curve.
\PP{The first contribution of the paper is theoretic, and consists in the form of the p.d.f. and c.d.f. of the \pof\ distribution.}
The results are obtained by marginalizing the joint density of the \pof\
observations, using integration and special functions, for any \pof\
horizon $h=1,2,\dots$.
The formulae are valid for stationary as well as non-stationary
GARCH(1,1) processes.

The 2-step-ahead \pof\ distribution is obtained without imposing
constrains on the values of the $\alpha$ and $\beta$ coefficients. For the
$h$-step-ahead \pof\ distribution with $h\geq 3$, a condition on $\beta$ is
required to guarantee integrability, which depends on how the coefficients are related to the last observed value of the volatility $\sigma_1^2$. \PP{Unless the last observed value of volatility $\sigma_1^2$ is relatively low, a sufficient condition for integrability is $\beta$ larger than $0.5$.
Another sufficient condition that is independent of the
last observed value of volatility $\sigma_1^2$
is $\beta$ larger than $0.61803$.}

\PP{As suggested by one referee, one could wonder how frequently the conditions $\beta \geq 0.5$
and $\beta \geq 0.61803$ are satisfied by empirical estimates in practice.
While this is ultimately an empirical question that depends on the data at hand, one could consider the estimates in
\cite{Bampinas2018}, who fit GARCH(1,1) models to all stock returns in the  S\&P 1500 universe, from January 2008 to December 2011, with daily observations.
Their Table 1 reports a mean value of $\hat{\beta}$ of $0.855$ (median $0.887$) with standard deviation of $0.154$.
If the empirical distribution of these estimates were Gaussian, the predicted frequency of $\beta$ estimates below $0.61803$ (respectively $0.5$) would be 6.2\% (respectively 1.1\%).
Typical values for $\beta$ estimates reported in textbooks and in empirical finance literature are also above $0.8$. On this basis, one would expect $\beta \geq 0.5$ or
$\beta \geq 0.61803$ not to be binding restrictions in most practical applications.\footnote{Table 1 in \cite{Bampinas2018} reports a skewness of $-2.670$ and kurtosis of $12.75$, which indicate that the distribution is far from normal; unfortunately, no min or max values are reported in the table. Typical $\hat{\beta}$ values in textbooks and in empirical finance literature can be found in \cite{Linton2019} Table 11.4, \cite{Tsay2010} page 136, \cite{Francq2010} page 262, \cite{Fabozzi2008} page 694; they are all greater than $0.8$ for daily, weekly and monthly returns. }
}

\PP{The assumption of Gaussianity of the innovations is central in the derivations in the paper, which employs analytical integration and special functions that are specific for this distribution. The approach in the derivation of the \pof\ distribution is expected to be amenable to extensions to non-Gaussian symmetric distributions of innovations. These extensions are not trivial and are not considered in the present paper.}

\PP{A first application of the present results is the characterization of} the degree of non-Gaussianity of the \pof\ distribution. This problem is relevant in practice, because the Gaussian density is often used as an approximation of the true \pof\ density, see e.g. \cite{Lau2010}, page 1322.
The exact formuale in this paper demonstrate that the \pof\ density can be very far from normal. However, for some parameter values encountered in applications, the discrepancy of the \pof\ density from the Gaussian density can be small, and this fact would support to the Gaussian approximation. Note that it is the analytic form of the \pof\ density derived in this paper that allows to measure this discrepancy.

\PP{A different question in this first area of application is the association of the coefficients of the GARCH with both the shape of the \pof\ distribution and the shape of the stationary distribution, when this exists. Under stationarity, the predictive distribution converges to the stationary distribution as the number of steps increases. The} tail behavior of the stationary distribution of the Gaussian GARCH(1,1) has been studied extensively, see \cite{Mikosch2000} and \cite{Davis2009a}. The tails
of the stationary distribution of both the volatility and of the GARCH
process $x_t$ are of Pareto type, $\Pr(x_t > u)\approx c u^{-2\kappa}$ say.
These properties are based on results for
random difference equations and renewal theory obtained in \cite{Kesten1973}
and \cite{Goldie1991}.


The \pof\ density is found to resemble a Gaussian density (with appropriate variance) for high values of $\beta
/ \alpha$, and far from it for low values of it.
Similarly, large values of $\beta / \alpha$ are found to be associated with higher values of $\kappa$, i.e. a Pareto stationary distribution with more moments (the Gaussian has all moments).

\PP{A second application of the explicit formulae in this paper is to compute exact tail}
probabilities and functionals, such as Expected Shortfall, that measure risk
when the underlying asset return is generated by a Gaussian GARCH(1,1). \PP{This improves on existing methods based on approximations or Monte Carlo (MC) simulations combined with (non-parametric) estimation techniques.}

The so-far unknown analytic form of the \pof\ density of a GARCH has led econometricians
to look for alternative approximate solutions. \cite{Alexander2013} have resorted to
approximations based on the first 4 moments of the \pof\ distribution;
\cite{Baillie1992} use a Cornish-Fisher expansion and a
Johnson SU distribution \PP{using the first 4 moments of the GARCH(1,1) to fit the distributions}.

Alternative methods for estimating the \pof\ density and
risk measures such as the Value at Risk and the Expected Shortfall
rely on MC simulations of the underlying GARCH processes, see e.g. \cite{Delaigle2016}.
All these MC methods implicitly assume that the density and that the unknown functionals are finite.\footnote{\cite{Delaigle2016}
proposed a non-parametric root-$n$ consistent estimator of the stationary distribution of the (log-)volatility process where $n$ is sample size.}

\PP{In this domain, the present exact formulae are key to prove that the density and the unknown functionals are finite, which is pre-requisite for MC methods to work.}
\PP{It must be noted, however, that MC methods have an additional layer of uncertainty -- associated with MC estimation -- that the exact methods proposed in this paper bypass entirely. Specifically, MC estimation of \pof\ functionals results in confidence intervals with positive length and coverage probability $1-\eta<1$; their counterpart based on the exact results of this paper can be represented by confidence intervals with length equal to 0 and coverage probability 1. Hence the exact methods in this paper are qualitatively superior to those based on MC methods, in addition to being much more parsimonious in terms of required calculations.}

\PP{A third application of} the exact formulae in this paper \PP{is to provide uncertainty} intervals for (functionals of) the \pof\ distribution \PP{that reflect in-sample estimation uncertainty.} This allows one to map estimation uncertainty onto forecast uncertainty for the risk measures and to construct the associated forecast intervals that have a pre-specified (asymptotic) coverage level. For instance, one could predict the Expected Shortfall to lie in an interval $(\ES_{l}, \ES_{u}) $ with 95\% confidence level, where the uncertainty reflects in-sample estimation uncertainty on the GARCH parameters.
\PP{As already discussed for the second area of application, the exact methods in this paper lead to results  that are structurally different from the ones based on MC methods. In fact the latter involve an additional layer of MC estimation uncertainty associated, which is avoided by the exact methods in the present paper.}

The rest of the paper is organised as follows. Section \ref{sec_appr} describes the general approach for the derivation of the \pof\ distribution. Section \ref{sec_main} states the main theoretical results. Section \ref{sec_form} discusses the degree of non-Gaussianity of the \pof\ distribution and compares the \pof\ distribution with the tails of the stationary distribution when this exists; this is the first area of application discussed above.
Section \ref{sec_appl} discusses the second area of application, i.e. how to apply the present exact results to the calculation of the Value at Risk and of the Expected Shortfall, and compares the obtained results with alternative estimators based on MC methods. Section \ref{sec_forecast} discusses the third area of application
of the present analytical formulae
to the construction of forecast intervals for risk measures that reflect in-sample estimation uncertainty. Section \ref{sec_conc} concludes. Proofs are collected in the Appendix. 


\section{The \pof\ density}
\label{sec_appr}

This section summarises the construction used to derive the \pof\ density
as an integral, involving a product of densities of the innovations.
Consider the asymmetric GJR-GARCH(1,1)\
\begin{equation}\label{eq_GARCH11}
x_{t}=\sigma _{t}\varepsilon _{t},\qquad \sigma _{t}^{2}=\omega +\alpha
_{t-1}x_{t-1}^{2}+\beta \sigma _{t-1}^{2},\qquad \alpha _{t}:=\alpha
+\lambda 1_{x_{t}<0} = \alpha
+\frac{\lambda}{2} (1- \varsigma_{t})
\end{equation}%
where $\omega, \alpha, \beta >0$, $\lambda \geq 0$ and
$1_{x_{t}<0} = \frac{1}{2}(1-\varsigma_{t})$ is the
indicator function for the event $x_{t}<0$, and
$\varsigma _{t}:=\sgn(\varepsilon _{t})=\sgn(x_{t})$ is the sign of $%
\varepsilon _t $ or $x_{t}$; these signs are the same
because $\sigma _{t}>0$.
The sequence $\{\varepsilon _{t}\}$ is
assumed to be i.i.d., centered around zero and with Gaussian
p.d.f. $f_{\varepsilon }(\epsilon):=g(\epsilon^2):=(2\pi )^{-\frac{1}{%
2}}\exp (-\epsilon^{2}/2)$.


Time $t=0$ is taken to be the starting time of the \pof, and it is assumed that one wishes to predict $x_{h}$ for some $h=1,2,3,\dots$, conditional on the information set at time $t=0$, which contains $x_0$ and $\sigma_0$; the conditioning on $x_0$ and $\sigma_0$ is not explicitly included in the notation of the \pof\ distribution for simplicity.
\footnote{The information set at $t=0$ containing contains $x_0$ and $\sigma_0$ is consistent with observing $x_t$ from
minus infinity to time 0 under stationarity. Note also that, because $x_0$
and $\sigma_0^2$ are observed, also $\sigma_1^2$ is observed.}

Throughout the paper the values taken by the random variables $x_{t}$, and $z_{t}:=x_{t}^{2}$ are denoted $u_{t}$ and $w_{t}$ respectively, or
sometimes $u$ and $w$ for simplicity, when this may not cause ambiguity.

The next Lemma reports consequences of the
symmetry of the one-step-ahead density $g$ on relevant conditional p.d.f.s.
In the Lemma, the following notation is used:
$\vz:=(z_{1},\dots ,z_{h-1})^{\prime }$, $\vvarsigma:=(\varsigma
_{1},\dots ,\varsigma _{h-1})^{\prime }$, where
$\vw:=(w_{1},\dots ,w_{h-1})^{\prime }$, $\vs:=(s_{1},\dots ,s_{h-1})^{\prime }$
denote values of $\vz$ and $\vvarsigma$.
\PP{The density of a random variable $x$ evaluated at $u$ is indicated as $f_x(u)$, and similarly $f_{x |\varsigma} (u|s)$ indicates the conditional density for $x$ given $\varsigma$, evaluated at $x=u$ and  $\varsigma=s$.}

\begin{lemma}[Densities]\label{lemma01}
For symmetric $f_{\varepsilon }(\epsilon)=g(\epsilon^2)$, the p.d.f. $f_{x_{t}}(\cdot)$ is symmetric, i.e.
$f_{x_{t}}(u)=f_{x_{t}}(-u)$, $u\in\SR$, and it is related to the p.d.f. of $z_{t}$ in the following way
\begin{equation} \label{eq_predictive_Gauss_xt}
f_{x_{t}}(u) =f_{z_{t}}(u^{2})\left\vert u\right\vert .
\end{equation}
Moreover, $\mathrm{\Pr }(\varsigma _{t}=\pm 1)=\frac{1}{2}$ and
one has
\begin{equation} \label{eq_decompose}
f_{\vz,z_{h} |\vvarsigma} (\vw,w_{h}|\vs)= \prod_{t=1}^{h}
\left(w_{t} \sigma_t^2 \right)^{-\frac{1}{2}} g\left( \frac{w_{t} }{\sigma_t^2} \right)
\end{equation}
where $\sigma^2_t$ depends on $w_{t-j}$ $($the value of $z_{t-j}=x_{t-j}^2$$)$ and $s_{t-j}$ $($the sign of $x_{t-j}$$)$ for $j=1,\dots,t-1$ via \eqref{eq_GARCH11}.
\end{lemma}

Next denote the set of all possible $h-1$ sign vectors $\vvarsigma$ by $\calS$, $\#\calS=2^{h-1}$.
Densities are first computed conditionally on $\vvarsigma$ and later they are marginalized
with respect to it. Here, conditioning on $\vvarsigma$ is relevant only for the GJR case $\lambda \neq 0$.

The basic building block is given by the expression in \eqref{eq_decompose}. This density can be marginalised with respect to $\vz$ as follows
\begin{equation}
f_{z_{h}|\vvarsigma}(w_{h}|\vs)=\int_{\mathbb{R}_{+}^{h-1}}f_{\vz,z_{h}|\vvarsigma%
}(\vw,w_h|\vs)\mathrm{d}\vw. \label{eq_integral_I}
\end{equation}%
Finally, $f_{z_{h}|\vvarsigma}(w_{h}|\vs)$ can be marginalised with respect to the signs $\vvarsigma$ using the mutual independence of the signs $\varsigma _{t-j}$ and the fact that $\mathrm{\Pr }(\varsigma _{t}=\pm 1)=\frac{1}{2}$
for all $t$, due to the symmetry of $g$. One hence finds
\begin{equation}
f_{z_{h}}(w_{h})=\sum_{\vs}~f_{z_{h}|\vvarsigma}(w_{h}|\vs)\mathrm{\Pr }(\vs%
)=2^{-h+1}\sum_{\vs}~f_{z_{h}|\vvarsigma}(w_{h}|\vs) \label{eq_fz_overall}
\end{equation}%
where the sum $\sum_{\vs}$ is over $s_{j}\in \{-1,1\}$, for $j=1,\dots ,h-1$.
The \pof\ density $f_{z_{h}}(w_{h})$ is found by combining \eqref{eq_fz_overall}, \eqref{eq_integral_I}, \eqref{eq_decompose}, \eqref{eq_predictive_Gauss_xt}.

The next Lemma reports a recursion for the volatility process, that turns out to be useful when solving the integral in \eqref{eq_integral_I}. In the Lemma,
the following notation is used:
for $t=1,\dots ,h-1$,
let $y_{t}:=\alpha _{t}z_{t}/(\beta \sigma
_{t}^{2})=\alpha _{t}x_{t}^{2}/(\beta \sigma _{t}^{2})=\alpha
_{t}\varepsilon _{t}^{2}/\beta $ and
$\vy :=(y_{1},\dots ,y_{h-1})^{\prime }$, where $\vv:=(v_{1},\dots
,v_{h-1})^{\prime }$ denotes a value of $\vy$.

\begin{lemma}[Volatility and transformations]\label{Lemma_2}
The volatility process can also be written
\begin{equation}\label{eq_vola}
\sigma _{h}^{2}=\omega +(1+y_{h-1})\beta \sigma _{h-1}^{2}\qquad y_{h}:=\frac{%
\alpha_h }{\beta }\varepsilon _{h}^{2}.
\end{equation}%
For $h\geq 2$, $\sigma _{h}^{2}$ has the following recursive expression in
terms of $y$'s
\begin{equation}
\sigma _{h}^{2} =\omega +\left( 1+y_{h-1}\right) \left\{ \omega \beta +\left(
1+y_{h-2}\right) \left( \dots \left( \omega \beta ^{h-2}+\left(
1+y_{1}\right) \beta ^{h-1}\sigma _{1}^{2}\right) \right) \right\}
\label{eq_sigma2_h}
\end{equation}%
with $\sigma _{1}^{2}=\omega +\beta \sigma _{0}^{2}+\alpha _{0}x_{0}^{2}$,
which is measurable with respect to the information set at time $0$.
Moreover, one has
\begin{equation}\label{eq_last_I}
f_{z_{h}|\vvarsigma}(w_{h}|\vs)=\left( \frac{\gamma _{h}}{w_{h}}\right) ^{%
\frac{1}{2}}\int_{\mathbb{R}_{+}^{h-1}}\prod_{t=1}^{h-1}\left( v_{t}^{-\frac{%
1}{2}}g\left( \frac{\beta }{\PP{a} _{t}}v_{t}\right) \right) \cdot \sigma
_{h}^{-1}g\left( \frac{w_{h}}{\sigma _{h}^{2}}\right) \mathrm{d}\vv,
\end{equation}
where
$\gamma _{h}:=\beta
^{h-1}/(\prod_{t=1}^{h-1}\PP{a}  _{t})$ and \PP{$a _{t} = \alpha + \frac{1}{2}\lambda (1-s_t)$
is the value of $\alpha _{t}$ corresponding to $\varsigma_{t}=s_t$.}
\end{lemma}

\section{Main results}
\label{sec_main}
The main results are summarised in Theorem \ref{theorem_1} below. Before
stating it, an auxiliary assumption is introduced. Define $\theta :=\omega /2\sigma
_{1}^{2}$, \PP{and note that this is bounded by 0 and $\frac{1}{2}$ as $\alpha$ and $\beta$ vary, $0<\theta \leq \frac{1}{2}$. Moreover define the following function of $\theta$: $\myb:=\myb(\theta):=-\theta +\sqrt{\theta ^{2}+2 \theta }$, which is used in the next Assumption.}

\begin{ass}
\label{ass_1}$\;$ \newline \vspace{-0.5cm}

\begin{itemize}
\item[a.] For $h=3$, let $\beta \geq \myb$;

\item[b.] For $h>3$ let $\beta \geq \max (\frac{1}{2},\myb)$.
\end{itemize}
\end{ass}

\noindent It can be noted that $\sup_{\theta}\myb(\theta)=\lim_{\theta\rightarrow \frac{1}{2} }\myb(\theta)=\frac{-1+\sqrt{5}}{2}\approx 0.618\,03$, as $%
\sigma _{1}^{2}>\omega $.
In Figure \ref{fig:h9a}, the area above the curve represents the
set $\beta \geq \max (\frac{1}{2},\myb)$ for $0<\theta \leq
\frac{1}{2}$.

\begin{figure}[ht]
\caption{$\protect\myb$ as a function of $\protect\theta:=\omega /(2 \sigma_1^2)$. Blue line: $\myb:=\myb(\theta):=-\theta +\sqrt{\theta ^{2}+2 \theta }$. Shaded area: region $\beta \geq \max (\frac{1}{2},\myb)$, see Assumption \ref{ass_1}.}
\label{fig:h9a}
\centering
\begin{tikzpicture}[scale=1] 
\begin{axis}[
axis lines=middle,
axis line style={-},
ymajorgrids=true,
xmajorgrids=true,
x label style={at={(current axis.right of origin)},anchor=north, below=5mm},
y label style={at={(current axis.above origin)},anchor=west, left=15mm, below=5mm},
xlabel={$\theta$},
ylabel={$\beta$},
xmin=0,
xmax=0.5,
ymin=0,
ymax=1,
xtick={0,0.125,0.25,0.375,0.5},
ytick={0,0.2,...,1},
x tick label style={
        /pgf/number format/.cd,
        fixed,
        fixed zerofill,
        precision=3,
        /tikz/.cd
    }
]

\draw[fill=gray!15] (0,0.5) -- (0.25,0.5) --
plot[smooth, samples=100, domain={0.25}:0.5] (\x,{-(\x)+sqrt((\x)^(2)+2*(\x))}) --
plot[smooth, samples=2, domain=0.5:0] (\x,{1}) -- cycle;

\draw[line width=2pt,color=blue,smooth,samples=100,domain=0:0.5] plot(\x,{-(\x)+sqrt((\x)^(2)+2*(\x))});
\end{axis}
\end{tikzpicture}
\end{figure}
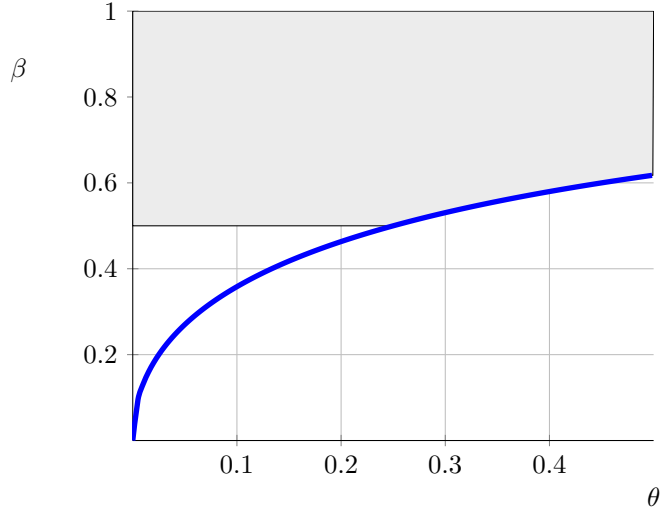

\PP{Note that $\myb$ depends on the ratio $\theta = \omega / (2 \sigma_1^2)$, where $\sigma_1^2$ is the last known value in the information set. Relatively large values of $\sigma_1^2$ correspond to $0<\theta\leq\frac{1}{8}$, moderate values to $\frac{1}{8}<\theta\leq\frac{1}{4}$ and small values to $\frac{1}{4}<\theta\leq\frac{1}{2}$. Note in fact, for instance, that $\theta\leq\frac{1}{4}$ corresponds to $\omega \leq \sigma_0^2 (\beta+\alpha \varepsilon_0^2)=\alpha x_0^2+ \beta\sigma_0^2$, an inequality one expects to be frequently valid.
Note that $\myb(\frac{1}{8})=0.39039$, $\myb(\frac{1}{4})=\frac{1}{2}$, $\myb(\frac{1}{2})=0.61803$, so that
Assumption \ref{ass_1} requires
$\beta>0.39039$ (respectively $0.5)$ when the last observed volatility is relatively large (respectively moderate) for the results in the paper for $h=3$ to hold. Only low values of the last observed volatility correspond to $\beta>\frac{1}{2}$.}

\PP{For $h>3$, $\beta \geq \max (\frac{1}{2},\myb)$. Hence, unless the last observed volatility $\sigma_1^2$ is very low i.e. $\frac{1}{4}<\theta\leq\frac{1}{2}$, the sufficient condition (which could be possibly further analytically improved) requires $\beta \geq \frac{1}{2}$.}

In Theorem \ref{theorem_1} below, $\Psi $ is the confluent
hypergeometric function of the second kind, also known as Tricomi function,
see \cite{Abadir1999} and \cite{Gradshteyn2007}, section 9.21, whose integral representation is,
\begin{equation}
\Psi(a;c;z) = \frac{1}{\Gamma(a)}\int_{\mathbb{R}_{+}}\exp \left( -zt\right)
t^{a-1}\left( 1+t\right) ^{c-a-1}\mathrm{d}t \label{eq:tricomi}
\end{equation}
with $\Re(z)>0, \Re(a)>0$.

\PP{The $\Psi$ function is used to define the following quantities:
\begin{align}\nonumber
A_{h}(r) &=\left( \beta \sigma _{1}^{2}\right) ^{-\frac{1}{2}}\pi ^\frac{h-1}{2}  \sum_{k_{1},\dots,k_{h-2}}\binom{r}{k_{1}}\binom{r-k_{1}}{k_{2}}\cdots
\binom{r-K_{h-3}}{k_{h-2}}\omega ^{K_{h-2}}\beta ^{(h-2)r-U_{h-2}} \cdot \\ &  \quad \cdot \prod_{t=1}^{h-2}
\Psi \left( \frac{1}{2},r-K_{t}+\frac{3}{2};\frac{\beta}{2\alpha_{h-t}}\right)
\Psi \left( \frac{1}{2},r-K_{h-2}+\frac{3}{2};\frac{\omega+\beta\sigma_1^{2}}{2\alpha_1\sigma_1^{2}}\right),
\label{eq_long_before}
\end{align}
where $K_0:=0$, $K_i:=\sum_{t=1}^{i}k_i$, $U_{j}:=\sum_{i=1}^{j}K_{i}=\sum_{i=1}^{j} (j-i+1) k_{i}$.
The multiple sum is defined for $h\geq3$ as
$\sum_{k_{1},\dots,k_{h-2}}:=\sum_{k_{1}=0}^{r}\sum_{k_{2}=0}^{r-k_{1}}\cdots
\sum_{k_{h-2}=0}^{r-K_{h-3}}$, where the individual sums extend to $\infty$ if $r\notin \mathbb{N}$. For $h=2$
the sum $\sum_{k_{1},\dots,k_{h-2}}$ and the product $\prod_{t=1}^{h-2}$ are empty and
\eqref{eq_long_before} reduces to
$A_{2}(r) =\left( \beta \sigma _{1}^{2}\right) ^{-\frac{1}{2}}
\pi ^\frac{1}{2}\Psi \left( \frac{1}{2},r+\frac{3}{2};\frac{\omega+\beta\sigma_1^{2}}{2\alpha_1\sigma_1^{2}}\right)$.}


\begin{thm}[GARCH(1,1) \pof\ density]
\label{theorem_1}Assume that $\varepsilon _{t}$ are i.i.d. $\rN(0,1)$ and let
Assumption $\ref{ass_1}$ hold; then one has, for $h\geq 2$, \PP{$w_{h}\geq 0$ and $-\infty < u_{h}<\infty$}
\begin{align}
f_{z_{h}}(w_{h}) &=(2\pi )^{-\frac{h}{2}}w_{h}^{-\frac{1}{2}} \sum_{j=0}^{\infty }\frac{1}{j!}
\PP{(-\rho_w)^{j}} c_{j},\qquad \PP{\rho_w:=\frac{w_{h}}{2(\omega+\beta\sigma_1^2)}}
\label{eq_predictive_Gaussian_2} \\
f_{x_{h}}(u_{h}) &=f_{z_{h}}(u_{h}^{2})\left\vert u_{h}\right\vert =
(2\pi )^{-\frac{h}{2}}\sum_{j=0}^{\infty }\frac{1}{j!}
\PP{(-\rho_u)^{j}}
c_{j}, \qquad \PP{\rho_u:=\frac{u_{h}^2}{2(\omega+\beta\sigma_1^2)}}
\label{eq_predictive_Gaussian}
\end{align}%
where $c_{j}:=2^{-h+1}\sum_{\vs\in \calS}c_{j,\vs}$ and
$c_{j,\vs}$ is defined as
$
c_{j,\vs}:= \gamma_h^{\frac{1}{2}}A_h\left(-j- \frac{1}{2} \right),
$
where $A_h(\cdot)$ is defined in \eqref{eq_long_before} and $\gamma_h$ in Lemma $\ref{Lemma_2}$.
The expressions in \eqref{eq_predictive_Gaussian_2} and \eqref{eq_predictive_Gaussian} are absolutely summable for any finite $w_h$ or $u_h$.
\end{thm}

\begin{proof}
See Appendix.
\end{proof}
Observe that the expression for $c_{j,\vs}$ does not depend on the points of evaluation $w_h$ or $u_h$, and hence
the $c_{j,\vs}$ coefficients can be computed only once for the whole densities.
One can prove, see Lemma \ref{lemma-Psi} in the Appendix, that $\Psi(a;c;z)\rightarrow 0$ for $c\rightarrow-\infty$. This implies that the terms
\PP{$\Psi \left( \frac{1}{2},r-K_{t}+\frac{3}{2};z\right)$
converge to 0 for large $k_t$ in the product in \eqref{eq_long_before}.}

Note also that for $h=2$, equation (\ref{eq_predictive_Gaussian})
holds for any value of $\beta $, while for $h=3$ it holds if and only if $%
\beta \geq \myb$. For $h>3$, the validity of the (\ref%
{eq_predictive_Gaussian}) is guaranteed by the sufficient condition $\beta
\geq \max (\frac{1}{2},\myb)$, which is, however, not necessary.

The line of proof of Theorem \ref{theorem_1} is the following:\ for $h=2$
the integral is solved by substitution and by using equation (\ref%
{eq:tricomi}). 
For $h\geq 3$, subsequent (negative) binomial expansions of expression (\ref%
{eq_sigma2_h}) for $\sigma _{t}^{2}$ are required, whose validity is ensured
by the inequality
\begin{equation*}
\omega \left( 1-\sum_{i=1}^{j-1}\beta ^{i}\right) \leq \beta ^{j}\sigma
_{1}^{2}, \qquad j \geq 2  \label{eq_ineq}
\end{equation*}%
which is satisfied under Assumption \ref{ass_1}, see Lemma \ref%
{Lemma_beta_cond} in the Appendix.

Immediate consequences of Theorem \ref{theorem_1} are collected in the
following corollary.

\begin{coro}[C.d.f. and moments]
\label{Corollary_2}The \pof\ c.d.f.s of $z_{h}$
and $x_{h}$ are given by \PP{the following expressions for $h\geq 2$, $w_{h}\geq 0$ and $-\infty < u_{h}<\infty$}
\begin{align}\nonumber
F_{z_{h}}(w_{h}) &=(2\pi)^{-\frac{h}{2}}\sum_{j=0}^{\infty }
\frac{w_{h}^{\frac{1}{2}}}{j!\left( j+\frac{1}{2}\right) } (- \rho_w)^j c_{j}, \qquad \PP{\rho_w:=\frac{w_{h}}{2(\omega+\beta\sigma_1^2)}}
\\
\label{eq_predictive_cdf}
F_{x_{h}}(u_{h}) &=
\frac{1}{2}+(2\pi)^{-\frac{h}{2}}\sum_{j=0}^{\infty }
\frac{u_{h}}{j!\left( 2j+1\right) }
 (- \rho_u)^j c_{j}, \qquad \PP{\rho_u:=\frac{u_{h}^2}{2(\omega+\beta\sigma_1^2)}}
\end{align}%
with \PP{$0$ odd moments for $x_{h}$} and even moments%
\PP{\begin{equation}\label{eq_even_moments}
\E(x_{h}^{2m})=\E(z_{h}^{m})=
2^{m-\frac{3}{2}(h-1)}\pi ^{-\frac{h%
}{2}}\Gamma \left( m+\frac{1}{2}\right)
(\omega+\beta \sigma_1^2)^{m + \frac{1}{2}}
\sum_{\vs\in %
\calS}
\gamma _{h}^{\frac{1}{2}} A_h (m)\qquad m=1,2,\dots
\end{equation}%
where $\gamma_{h}$ and $A_{h}(m)$ depend on $\vs$,
see their definitions in Lemma $\ref{Lemma_2}$ and in \eqref{eq_long_before}.
}
%
\end{coro}

Note that \PP{$A_{h}(m)$} in the
moments calculations are made of finite sums extending to $m$,
involving the Tricomi functions, which do not fall in the logarithmic case as in Theorem \ref{theorem_1}; see \cite{Abadir1999} for the logarithmic case.
In fact, $m-k\in \left\{ 0,1,\dots ,m\right\} $
implies that%
\begin{align*}
\Psi \left( \frac{1}{2};\frac{3}{2}+m-k;\xi \right) &=\frac{\Gamma \left(
\frac{1}{2}+m-k\right) }{\sqrt{\pi }}\xi ^{-\frac{1}{2}-m+k}\,_{1}F_{1}%
\left( -m+k;\frac{1}{2}-m+k; \xi \right) \\
&=\frac{\Gamma \left( m+\frac{1}{2}\right) }{\sqrt{\pi }k!\binom{m-\frac{1}{%
2}}{k}}\xi ^{-\frac{1}{2}-m+k}\sum_{j=0}^{m-k}\frac{\binom{m-k}{j}}{\binom{-%
\frac{1}{2}+m-k}{j}}\frac{\xi ^{j}}{j!}
\end{align*}%
is a finite sum, see \cite{Abadir1999}, which is proportional to the generalized Laguerre polynomial
$L_{m-k}^{(-1/2-m+k)}(\xi )$, where
$L_{i}^{(a )}(\xi ):= \sum_{k=0}^{i}\left( a +1+k\right)_{i-k}\left( -i\right) _{k}\frac{\xi ^{k}}{k!} $
see e.g. \cite{Abramowitz1964} Chapter 22.
\PP{For the moments of a GARCH(1,1), one can compare \eqref{eq_even_moments} with equations (34) and (35) in \cite{Baillie1992}.}


Some standardized densities of $x_{h}$ and the corresponding right tails are
plotted in Figure \ref{fig:n1} for $h=1,2,3,4$. The curve $h=1$ is the
standard Gaussian. Computations for Figures \ref{fig:n1}, \ref{fig:ratio} and \ref{fig:n1_tail_2} were performed in Mathematica.\footnote{When $x$ has mean 0 and standard deviation $s$, the standardized variate is $z=x/s$, with density
$f_z(a) = s f_x(sa)$.}

Figure \ref{fig:ratio} shows the standardized \pof\ densities
for $h=2$ and values of $\beta/\alpha$ that range from to
8.5 ($\alpha = 0.1, \beta = 0.85$) to 1/8.5 ($\alpha = 0.85, \beta = 0.1$).
This figure shows that the deviations from the Gaussian case of the \pof\ density can be substantial; the \pof\ densities are more similar to a Gaussian when $%
\beta /\alpha $ is large.
Figure \ref{fig:n1_tail_2} shows the tails for the GJR-GARCH(1,1) case.

\begin{figure}[ht]
\caption{Prediction densities $f_{x_h}(u_h)$ (left panel) and zoom of
the right tails (right panel) for standardized $x_h$, $h=1,2,3,4$, $\protect%
\omega = 0.1, \protect\alpha = 0.1 , \protect\beta = 0.7, \protect\sigma^2_0
= 1; x^2_0 = 1, \protect\lambda =0$. }
\label{fig:n1}
\par
\begin{center}
\includegraphics[scale=0.6]{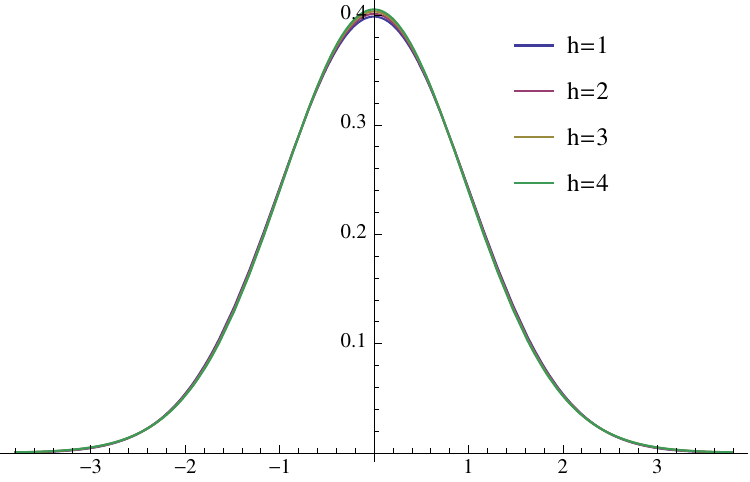} %
\includegraphics[scale=0.6]{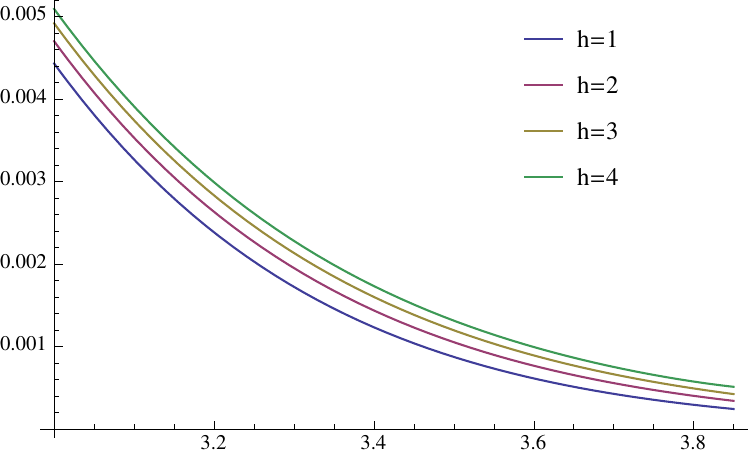}
\end{center}
\end{figure}

\begin{figure}[ht]
\caption{Prediction density $f_{x_2}(u_2)$ for standardized $x_2$, $\protect%
\omega = 0.1, \protect\sigma^2_0 = 1; x^2_0 = 1$ varying values of $(\alpha,\beta)$}
\label{fig:ratio}
\begin{center}
\includegraphics[scale=0.6]{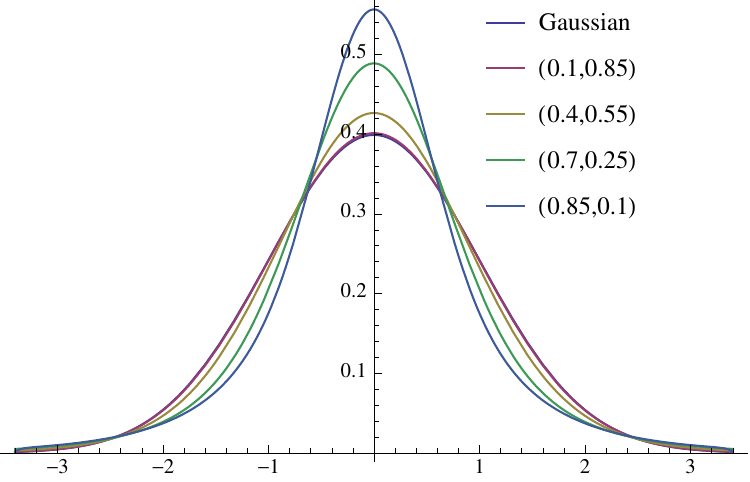} %
\includegraphics[scale=0.6]{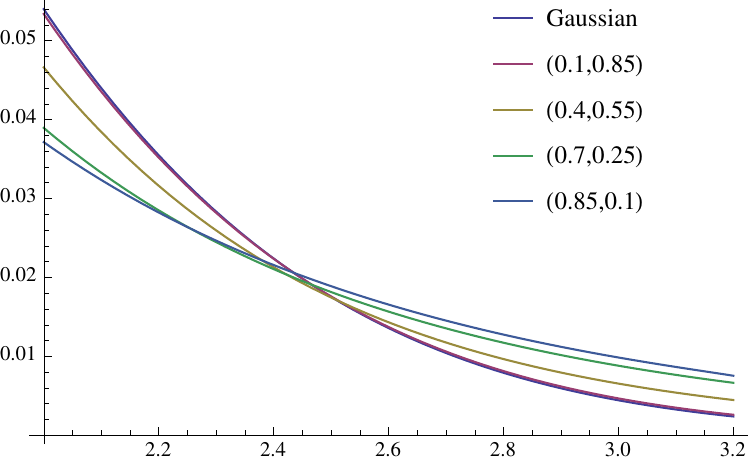}
\end{center}
\end{figure}

\begin{figure}[ht]
\caption{Right tail of $f_{x_h}(u_h)$ for standardized $x_h$, $h=1,2,3$, in
blue, red and green respectively, $\protect\omega = 0.25, \protect\alpha =
0.1 , \protect\beta = 0.7, \protect\sigma^2_0 = 1; x^2_0 = 1, \protect%
\lambda =0.2$ ($h=1$ is standard Gaussian) }
\label{fig:n1_tail_2}
\begin{center}
\includegraphics[scale=0.6]{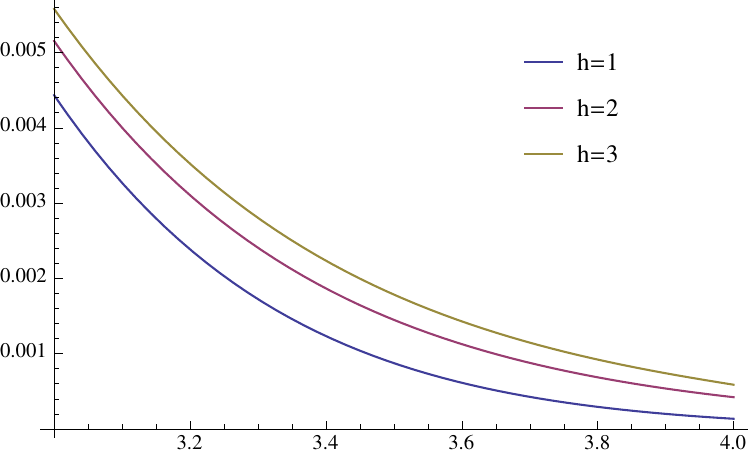}
\end{center}
\end{figure}

\PP{The formulae in Theorems \ref{theorem_1} and Corollary \ref{Corollary_2} are alternating in sign. While (absolutely) convergent, the associated series was found in practice to be ill-behaved numerically when $\rho_u:=u^2/(\omega+\beta\sigma_1^2)$ is very large, causing the oscillations in the terms of the series to become large before decreasing in amplitude toward zero, where `large' refers to the greatest floating point number handled by the computer.
Note that this is can be linked to large $u$ and/or small $\omega+\beta\sigma_1^2$.
This extreme behaviour implies accumulation of numerical errors, which can lead to inaccurate calculations of the \pof\ density.}

\PP{
\begin{example}[Numerical accuracy]\label{ex_num_acc}
One such case can be obtained for the density of $x_2$ in formula \eqref{eq_predictive_Gaussian}, $h=2$,
in the following way: select $u=4$ for $\omega =0.0000114, \alpha = 0.85, \beta = 0.14$, and choose $\sigma_0^2 = x_0^2 = \sigma^2:=\omega / (1-\alpha - \beta)=0.00114$. This results in $\rho_u = 53.\bar{3}$ for the standardized p.d.f. of $x_2/s$ with $s^2:=\omega + (\alpha + \beta)\sigma^2_1$. In this case, the oscillations of the terms in the series increase up to $\pm 4 \cdot 10^{20}$ around the $50^{th}$ term of the series, before oscillations decrease toward zero; the resulting series truncated after its first 100 terms gave the negative number $-2.9628 \cdot 10^{12}$. Calculations performed in MATLAB 2018a on an Intel i7 Windows 10 computer.
As a comparison, the same script applied to $u=2$ gave
$f_{x_2}(2)=0.03688432$
\end{example}
}

\PP{In order to address these numerical accuracy problems when $\rho_u=u^2/(\omega+\beta\sigma_1^2)$ is large, the following theorem presents a different set of formulae for the \pof\ density. This alternative set has the advantage to allow computations in the far tails of the density, at the price of a slightly higher implementation cost.}

\begin{thm}[Alternative formulae for the GARCH(1,1) \pof\ density]\label{theorem_2}
\PP{Under the same assumptions of Theorem $\ref{theorem_1}$ one has \PP{for $h\geq 2$, $w_{h}\geq 0$ and $-\infty < u_{h}<\infty$}
\begin{align}
f_{z_{h}}(w_{h}) &=(2\pi )^{-\frac{h}{2}}w_{h}^{-\frac{1}{2}} \e^{-\rho_w} \sum_{j=0}^{\infty }p_{j}(\rho_w)
 c^{\star}_{j}, \qquad \rho_w:= \frac{w_{h}}{2(\omega+\beta\sigma_1^2)}
\label{eq_predictive_Gaussian_2BIS} \\
f_{x_{h}}(u_{h}) &=f_{z_{h}}(u_{h}^{2})\left\vert u_{h}\right\vert =
(2\pi )^{-\frac{h}{2}}\e^{-\rho_u} \sum_{j=0}^{\infty }p_{j}(\rho_u)
c^{\star}_{j}, \qquad \rho_u:= \frac{u_{h}^2}{2(\omega+\beta\sigma_1^2)}.
\label{eq_predictive_GaussianBIS}
\end{align}%
For $h=2$,
$p_{j}(\cdot)$ is defined as
$p_{j}(\rho) := \rho^j / j!$,
$c_{j}^{\star}:=2^{-1}\sum_{\vs\in \calS}c_{j,\vs}^{\star}$ and
\begin{equation}\label{eq_cjsstar}
c_{j,s}^{\star}:=\pi ^{\frac{1}{2}}(\sigma _{1}^{2}\alpha _{1})^{-\frac{1}{2}%
}
\left( \frac{1}{2}\right)_{j}
\Psi\left( j+\frac{1}{2},1; \frac{\omega +\beta \sigma _{1}^{2}}{2\alpha _{1}\sigma _{1}^{2}}\right),
\end{equation}
where $(a)_j:= \prod_{i=1}^{j-1}(a+i)$ denotes Pochhammer's symbol, see \cite{Abadir1999}.
}

\PP{
For $h\geq 3$,
$p_{j}(\rho ):=(-1)^{j}L_{j}^{(-1)}(\rho )$ where $L_{j}^{(-1)}(\rho
):=\sum_{k=0}^{j}\left( k\right) _{j-k}\left( -j\right) _{k}\frac{\rho ^{k}}{%
k!}$ is a generalized Laguerre polynomial\footnote{%
$L_{j}^{(a)}(x)$ is the standard notation, see e.g. \cite{Abramowitz1964} Chapter 22.}, with the convention $\left( 0\right) _{0}:=1$;
moreover
$c_{j}^{\star}:=2^{-h+1}\sum_{\vs\in \calS}c_{j,\vs}^{\star}$ where
$c_{j,s}^{\star}$ is defined as
\begin{align}
&c_{j,\vs} ^{\star}:=\pi ^{\frac{h-1}{2}}\left( \sigma _{1}^{2}\right) ^{-\frac{1}{2%
}}\cdot  \nonumber\\
& \cdot  \sum_{k_{1}=0}^{\infty }\sum_{k_{2}=0}^{\infty
}\cdots \sum_{k_{h-2}=0}^{\infty }\binom{-\frac{1}{2}%
}{k_{1}}\alpha _{h-1}^{-\frac{1}{2}}\Psi \left( \frac{1}{2},j+1-k_{1};\frac{%
\beta }{2\alpha _{h-1}}\right) \cdot \beta ^{j(h-2)-U_{h-2}}\cdot \left(
\frac{\omega +\beta \sigma _{1}^{2}}{\omega }\right) ^{j-K_{h-2}}\cdot
\label{eq_c_r_h_varsigmaBIS3} \\
&\cdot  \prod_{t=2}^{h-2}\alpha _{h-t}^{-\frac{1}{2}}\binom{j-\frac{1}{2}%
-K_{t}}{k_{t}}\Psi \left( \frac{1}{2},j+1-K_{t};\frac{\beta }{2\alpha _{h-t}}%
\right) \alpha _{1}^{-\frac{1}{2}}\Psi \left( \frac{1}{2},j+1-K_{h-2};\frac{%
\omega +\beta \sigma _{1}^{2}}{2\alpha _{1}\sigma _{1}^{2}}\right)
\nonumber
\end{align}
with
$K_{0}:=0$, $K_{t}:=\sum_{i=1}^{t}k_{i}$,
$U_{j}:=\sum_{i=1}^{j}K_{i}=\sum_{i=1}^{j} (j-i+1) k_{i}$.
The expressions in \eqref{eq_predictive_Gaussian_2BIS}, \eqref{eq_predictive_GaussianBIS},
\eqref{eq_cjsstar},
\eqref{eq_c_r_h_varsigmaBIS3}
are summable for any finite $w_h$ or $u_h$.}
\end{thm}

\PP{The improved numerical performance of formula \eqref{eq_predictive_GaussianBIS}
is linked to the presence of the term $\e^{-\rho_u}$ when $\rho_u:= \frac{u^2}{2(\omega+\beta\sigma_1^2)}$ is large.
In fact for $u^2 \rightarrow \infty$ the term $\e^{-\rho_u}\rightarrow 0$, so that $\e^{-\rho_u}$ compensates the large terms of the type $\rho_u^j$ that appear in the sum for large $u^2$. For $u^2 \rightarrow 0$, the term $\e^{-\rho_u}\rightarrow 1$, so that $\e^{-\rho_u}$ does not influence the sum for small values of $u^2$. Note, moreover, that all the terms in the series \eqref{eq_predictive_Gaussian_2BIS} and \eqref{eq_predictive_GaussianBIS} are positive, so that there are no oscillations associated with different signs for the terms in the series.}

\PP{
\begin{example}[Numerical accuracy - continued]\label{ex_num_acc2}
In the same setup of Example \ref{ex_num_acc}, formula \eqref{eq_predictive_GaussianBIS} is numerically accurate.
In fact, all the terms in the series were found to be bounded by $2 \cdot 10^{-4}$, with value of the density equal to
$0.002953901$,
again using the first 100 terms of the series. Calculations were performed in the same environment as in Example \ref{ex_num_acc}.
As a comparison, the same script applied to $u=2$ gave
$f_{x_2}(2)=0.03688291$,
which agrees with formula \eqref{eq_predictive_Gaussian}
in
Example \ref{ex_num_acc} up to the 5th digit (discrepancy equal to $1.4017 \cdot 10^{-6}$).
\end{example}
}

\PP{The slightly higher implementation cost of formula \eqref{eq_predictive_GaussianBIS} is associated with the
presence of the generalised Laguerre polynomial in $p_j(\rho)$ for $h\geq3$. They are finite sums and add a moderate cost in terms of computations.}
\PP{Similar derivations to Corollary \ref{Corollary_2} can be performed on \eqref{eq_predictive_Gaussian_2BIS}, \eqref{eq_predictive_GaussianBIS}
to derive the corresponding c.d.f.s.
}

\section{Stationary distribution}
\label{sec_form}

The limit representation of the random variable $x_{h}$ in the stationary
case can be found in \citet{Francq2010} Theorem 2.1 page 24. The tail
behaviour of the limit distribution is reviewed in \citet{Mikosch2000} and %
\citet{Davis2009a}. The tails of the stationary distribution of both the
volatility and of the GARCH process $x_{t}$ are of Pareto type,
$\Pr(x_t > u)\approx c u^{-2\kappa}$ say, where $\kappa >0$ is a tail index.
These properties are based on results for random difference equations and
renewal theory obtained in \cite{Kesten1973} and \cite{Goldie1991}.

The tail index of the stationary distribution depends on the coefficient $%
\alpha $ and $\beta $ of the GARCH(1,1) process $x_{t}$ as well as on the
one-step-ahead distribution. Examples of the tail index are given in \cite%
{Davis2009a}; for Gaussian innovations, $\kappa =14.1$ for $\alpha =\beta
=0.1$, while $\kappa =1$ for $\alpha =1-\beta$. 

The index $\kappa$ is the unique solution of $\E((\alpha \varepsilon_t ^2 +\beta)^{\kappa} )=1$.
When $\kappa$ is an integer, the expression simplifies to
\begin{equation}\label{eq_DM10}
1= \E((\alpha \varepsilon_t ^2 +\beta)^{\kappa} ) = \beta ^{\kappa} \sum_{i=0}^{\kappa} \binom{\kappa}{i}\left(\frac{\alpha}{\beta}\right)^i \E(\varepsilon_t^{2i}),
\end{equation}
see \citet{Davis2009a} eq. (10). Substituting the moments $\E(\varepsilon_t^{2n})$ from the $\chi^2$ distribution, and assigning values to $\alpha /\beta $ over a grid of pre-specified values, one can solve \eqref{eq_DM10} for $\beta$, and hence for $\alpha = (\alpha /\beta ) \beta$. This allows to compute (values of) the surface $\kappa(\alpha,\beta)$.
Figure
\ref{fig:abk} reports the level curves of $\kappa(\alpha,\beta)$ as a function of $\alpha$ and $\beta$ obtained in this way. The figure also reports the lines where $\beta/\alpha$ is constant. It is seen that, for large values of $\beta / \alpha$, $\kappa$ and $\beta / \alpha$ increase roughly together. This association is not present for small values of $\beta / \alpha$.

\begin{figure}[ht]
\caption{Level curves of $\kappa$ as a function of $\alpha$ and $\beta$ in the Gaussian case. Dashed lines represent loci where $\beta / \alpha$ is constant.}
\label{fig:abk}
\begin{center}
\includegraphics[scale=0.8]{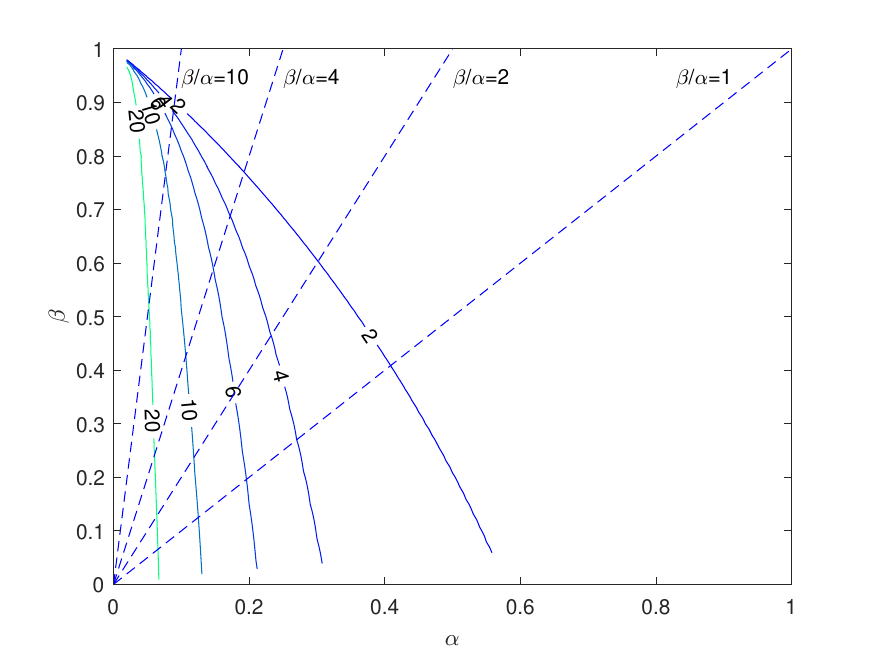}
\end{center}
\end{figure}

The relation between $\beta / \alpha$ and fat-tailedness of the \pof\ density for finite horizon $h$ can be illustrated using the case $h=2$. From
Theorem \ref{theorem_1},
\begin{equation*}
f_{x_{2}}(u_{2})= \frac{1}{\sqrt{2\pi }\tilde\sigma_2 } \sum_{j=0}^{\infty }%
\frac{1 }{j!} \left(-\frac{1}{2}\frac{u_2^2}{\tilde\sigma^2_2}\right)^{j}%
\sqrt{z} \Psi\left(\frac{1}{2}; 1-j;z \right)
\end{equation*}
where $\tilde \sigma _{2}^{2} = \omega + \beta \sigma _{1}^{2}$ and%
\footnote{The quantity $\tilde \sigma _{2}^{2} := \omega + \beta \sigma _{1}^{2}$ can
be interpreted as the minimum value that $\sigma^2_{2} = \omega +\left(
1+y_{1}\right) \beta \sigma _{1}^{2} $ can take, in the ideal case when $%
\alpha = 0$ (thus $y_1=0$) and $\sigma_1^2$ is given, i.e. $x_2 \sim
\textrm{N}(0,\tilde\sigma^2_2)$.}
\begin{equation*}
 z := \frac{\omega + \beta \sigma_1^2 }{2\alpha\sigma^2_1}
= \frac{\tilde\sigma_2^2}{2\alpha\sigma^2_1} = \frac{1}{2}\left(\frac{\omega}{\beta \sigma_1^2} + 1 \right) \frac{\beta}{\alpha}.
\end{equation*}
Hence when $ {\beta} / {\alpha}\rightarrow \infty$ one has $z\rightarrow \infty$ with $\sqrt{z} \Psi\left(\frac{1}{2}; 1-j;z \right) = 1+ O(|z|^{-1})$, see \cite{Abramowitz1964}, eq. 13.1.8, so that all the Tricomi functions $\Psi_j$, for varying $j$, tend to one.\footnote{This is unlike in the case for fixed $z$ where the sequence of $\Psi_j$ is decreasing to 0 for increasing $j$.}
As a result, when $ {\beta} / {\alpha}\rightarrow \infty$ the \pof\ distribution converges to a $\textrm{N}(0,\tilde\sigma^2_2)$.

One concludes that both for the \pof\ density for $h=2$ and for the stationary distribution, the fat-tailedness of the distributions is small for large values of ${\beta}/ {\alpha}$, unless $\alpha$ is very close to 0.

\section{Comparing exact formulae with simulation-based methods}
\label{sec_appl}
This section describes the application of the formulae in the previous section to the calculation of the Value at Risk and of the Expected Shortfall, comparing them with alternatives based on Monte Carlo. \PP{This comparison is made under Gaussianity and Assumption \ref{ass_1}, so that the formulae in the paper can be applied. 
The analysis in this and the remaining sections is for generic forecast horizon $h=2,3,\dots$, while illustrations are made for $h=2$ and $\lambda=0$ for simplicity and without loss of generality.}


Let $p$ be some tail probability, such as 5\%, and let $\VaR_{h,p}$ be the
Value at Risk, defined as the (negative) of the $p$ quantile of the
\pof\ distribution, i.e. $p=\Pr (x_{h}<-\VaR_{h,p})$. Let also $\ES_{h,p}$
indicate the corresponding expected shortfall, i.e. $\ES_{h,p}:=-\E%
(x_{h}|x_{h}<-\VaR_{h,p})$, following standard notation, see e.g. \cite{Francq2015}.

Observe that $\ES_{h,p}$ may fail to exist when the underlying density has Cauchy tails.
One implication of the exact results in Theorem \ref{theorem_1} of Section \ref{sec_main} is that for finite $h$ the \pof\ of $x_{h}$ has thinner-than-Cauchy tails, and hence $\ES_{h,p}$ exists; this appears to be a central issue for the application of $\ES_{h,p}$ as a measure of risk.

The following subsections show first how the exact formulae can be applied in this context, and next their relative advantage over methods based on MC methods. The same advantages discussed for the quantification of the Value at Risk and the Expected Shortfall apply more generally to other functionals of the \pof\ distribution, as well as to the
nonparametric estimation of the \pof\ distribution itself. For brevity, these latter cases are not discussed in this paper in detail.

The rest of the section refers to the standardized \pof\ distribution of $x_{2}$
when $\omega=1.14 \cdot 10^{-5}$, $\alpha=0.131007$, $\beta=0.845708$, $\lambda=0$; these values are the ML estimates on a AR(2)-GARCH(1,1) model for the weekly S\&P500 stock index return from 1950-2018 reported in Table 11.4 in \cite{Linton2019}.
These values of $\alpha$ and $\beta$ are very similar to the median estimates in Table 1 of \cite{Bampinas2018} for the set of individual S\&P 1500 daily returns. In the calculations $\sigma_1^2$ was set equal to $\omega / (1-\alpha-\beta)$.
For these parameters, standard double precision was found to be sufficient for $h=2$ for a range of $|x|<6$ in standardized units.

\subsection{Exact calculations}\label{sebsec_exact}
Both $\VaR_{h,p}$ and $\ES_{h,p}$ can be calculated using the exact formulae in this paper. This subsection combines the use of results in Section \ref{sec_main} with numerical techniques to illustrate applications of these results. This approach is chosen to
keep derivations as simple as possible, even when the analytical results of Section \ref{sec_main} could be extended to replace numerical integration.

Consider first $\VaR_{h,p}$; this can be found as the root of the function $F_{x_{h}}(u)-p$, where $F_{x_{h}}(u)$ is given in \eqref{eq_predictive_cdf}, using
root-finding algorithms like 
Newton's method -- see e.g. \cite{Press_2007},
Chapter 9 -- where
\begin{equation}\label{eq_QNewton}
u_{n+1}=u_{n}-\frac{F_{x_{h}}(u_n)-p}{f_{x_{h}}(u_n)}.
\end{equation}%
Here $F_{x_{h}}(u)$ is given in \eqref{eq_predictive_cdf} and $f_{x_{h}}(u)$ is given in
\eqref{eq_predictive_Gaussian}; this typically requires a handful of function evaluations.

Consider next $\ES_{h,p}$; one can write
\begin{align}\label{eq_ES_0}
\ES_{h,p} & =-\frac{1}{p}\int_{-\infty }^{-\VaR_{h,p}}u~f_{x_{h}}\left( u\right) \mathrm{d}%
u=\frac{\VaR_{h,p}}{p}~F_{x_{h}}( -\VaR_{h,p}) +\frac{1}{p}\int_{-\infty
}^{-\VaR_{h,p}}F_{x_{h}}\left( u\right) \mathrm{d}u \\
 & =\VaR_{h,p} +\frac{1}{p}\int_{-\infty
}^{-\VaR_{h,p}}F_{x_{h}}\left( u\right) \mathrm{d}u \nonumber
\end{align}%
where the second equality follows by integration by parts, and the third because $F_{x_{h}}( -\VaR_{h,p})=p$ by definition.%
\footnote{Note that, whenever $\E(x)$ exists, one has $\lim_{x\rightarrow -\infty} ( -xF ( x )  ) =0$; in fact for $-\infty <u<x$  one has%
\begin{equation*}
0\leq \lim_{x\rightarrow -\infty }\left( -xF\left( x\right) \right)
=\lim_{x\rightarrow -\infty }\left( -\int_{-\infty }^{x}xf\left( u\right)
\mathrm{d}u\right) \leq \lim_{x\rightarrow -\infty }\left( -\int_{-\infty
}^{x}uf\left( u\right) \mathrm{d}u\right) =0.
\end{equation*}
}
This integral
can be evaluated numerically using $f_{x_{h}}(u)$ in \eqref{eq_predictive_Gaussian}, or $%
F_{x_{h}}(u)$ in \eqref{eq_predictive_cdf}, employing quadrature methods (trapezoid), see \cite{Press_2007},
Chapters 4 and 13.

Table \ref{table_0} reports values of $\VaR_{2,p}$
using the reference values from Table 11.4 in \cite{Linton2019}.
The chosen algorithm in \eqref{eq_QNewton} was implemented in \textsc{Matlab}, using a tolerance value of $ 10^{-7}$ and avoiding to divide by $f$ when this is smaller than $10^{-14}$. Initial values of the iterations were chosen equal to the corresponding Gaussian quantiles.
Values of $F_{x_{h}}(u)$ were computed as in \eqref{eq_predictive_cdf} and $f_{x_{h}}(u)$ as in
\eqref{eq_predictive_Gaussian}, truncating sums at 100 terms.

Table \ref{table_0}
reports terminal values of the iterations, along with number of iterations and a comparison with the standard Gaussian distribution. Unsurprisingly, the Values at Risk are found to be close to the Gaussian quantiles. However, they are both smaller or larger than the Gaussian, depending on the value of $p$. The number of iterations needed was smaller than 5.

\begin{table}[tbp] \centering%
\caption{Values of $\VaR_{2,p}$ for $\omega=1.14 \cdot 10^{-5}$, $\alpha=0.131007$, $\beta=0.845708$, $\lambda=0$ using iterations \eqref{eq_QNewton}.}\label{table_0}
\begin{tabular}{lcccc}
\hline
$p$ & exact $\VaR_{2,p}$ & \# of iterations & Gaussian & ratio=Gaussian/exact \\
\hline
0.05	&	$1.6415$	&	3	&	$1.6449$	&	1.0020	\\
0.025	&	$1.9635$	&	3	&	$1.9600$	&	0.9982	\\
0.01	&	$2.3443$	&	3	&	$2.3263$	&	0.9924	\\
0.005	&	$2.6092$	&	4	&	$2.5758$	&	0.9872	\\
\hline
\end{tabular}%
\end{table}%

Table \ref{table_02}
reports values of $\ES_{2,p}$
using the standardized \pof\ distribution
with the same parameter values as in Table \ref{table_0}.
Numerical integration as in the last expression in \eqref{eq_ES_0} was performed using the standard function \texttt{integral} in \textsc{Matlab} with standard tolerance values; this uses global adaptive quadrature integration methods. Minus infinity was replaced in the calculations with $-6$.
Values of $F_{x_{h}}(u)$ were computed as in \eqref{eq_predictive_cdf} and $f_{x_{h}}(u)$ as in
\eqref{eq_predictive_Gaussian}, truncating sums at 100 terms.

\begin{table}[tbp] \centering%
\caption{Values of $\ES_{2,p}$ for $\omega=1.14 \cdot 10^{-5}$, $\alpha=0.131007$, $\beta=0.845708$, $\lambda=0$ using \eqref{eq_ES_0} and numerical integration.}\label{table_02}%
\begin{tabular}{lccc}
\hline
$p$ & exact $\ES_{2,p}$& Gaussian & ratio=Gaussian/exact \\
\hline
0.05	&	2.0745	&	2.0627	&	0.9943	\\
0.025	&	2.3620	&	2.3378	&	0.9898	\\
0.01	&	2.7121	&	2.6652	&	0.9827	\\
0.005	&	2.9612	&	2.8919	&	0.9766	\\
\hline
\end{tabular}%
\end{table}%

Table \ref{table_02} shows that the Expected Shortfall values are close to the Gaussian case, but systematically lower than them. In practice, the call to the  \texttt{integral} function was quicker than the computation of the $\VaR_{2,p}$ in the Table \ref{table_0}.

\subsection{Alternatives based on Monte Carlo}
Alternative methods to compute $\VaR_{h,p}$ and $\ES_{h,p}$ rely on MC simulations.
Simple MC solutions are reviewed here for comparison with the exact methods above. In order to estimate $\VaR_{h,p}$
and $\ES_{h,p}$, for replication $j=1,\dots,n$, one could generate pseudo random
numbers $( \varepsilon _{t,j}^{\ast } ) _{t=1}^{h}$ and construct
the corresponding values $( x_{t,j}^{\ast } ) _{t=1}^{h}$ using
recursion \eqref{eq_GARCH11}. Let $x_{h,j}^{\ast }$ be the $j$-th MC\ realization
of $x_{h}$ constructed in this way, \PP{and observe that $x_{h,j}^{\ast }$ are independent realisations across repetitions $j$ from the \pof\ distribution.}

Repeating this for $j=1,\dots ,n$, the sample
$( x_{h,1}^{\ast },\dots ,x_{h,n}^{\ast }) $ can be formed;\
let $ ( q_{1},\dots ,q_{n}) $ indicate the ordered values of $%
( x_{h,1}^{\ast },\dots ,x_{h,n}^{\ast })$, with $q_{1},\leq \dots
\leq q_{n}$. The MC quantile $q_{\lfloor n p\rfloor +1}$ can be used to estimate $-\VaR_{h,p}$, where
$\lfloor x\rfloor$ and $\lceil x\rceil$ indicate the round-down or round-up of $x$ to the nearest integer.\footnote{$\lfloor x \rfloor $ (respectively
$\lceil x\rceil$)
denotes the largest (respectively smallest) integer value less or equal
(respectively greater or equal)
to $x$.}

\begin{table}[tbp] \centering%
\caption{Number of replications $n$ for a MC confidence interval on $\VaR_{2,p}$, see eq. \eqref{eq_length_VaR}, \newline at given MC coverage level $1-\eta$.}\label{table_10}%
\begin{tabular}{lcccc}
\hline
		& $p=0.05$ & 	0.025 & 	0.01 & 	0.005	\\\hline
$\eta=0.05$ 	&	7.0710E+11	&	1.1604E+12	&	2.3680E+12	&	4.2104E+12	\\
$\eta=0.01$	&	1.2213E+12	&	2.0043E+12	&	4.0900E+12	&	7.2722E+12	\\
\hline
\end{tabular}%
\end{table}%

\PP{Observe here that the sample is a (pseudo) i.i.d. sample from the \pof\ distribution, and hence all results for i.i.d. samples apply on it.}
Standard results of quantiles based on the application of the central limit theorem to the MC empirical c.d.f., see e.g. \cite{Dudevicz1988} Theorem 7.4.21, imply that
\begin{equation*}
\sqrt{n}\left( q_{\left\lfloor n p\right\rfloor +1}+\VaR_{h,p}\right) \overset%
{w}{\rightarrow }\rN\left( 0,\frac{p(1-p)}{f_{x_{h}}^{2}(\VaR_{h,p})}\right) ,
\end{equation*}%
where $\overset{w}{\rightarrow }$ indicates weak convergence for $n\rightarrow \infty$.
Hence a MC\ large-$n$\ confidence interval for $\VaR_{h,p}$ using $%
q_{\left\lfloor np\right\rfloor +1}$ at level $\eta $ is given by $%
q_{\left\lfloor np\right\rfloor +1}\pm z_{1-\eta /2}\sqrt{p(1-p)}/(\sqrt{n}%
f_{x_{h}}(q_{\left\lfloor np\right\rfloor +1}))$ where $z_{b}$ is the $b$%
-quantile from the standard normal distribution. 
The length of the confidence interval for $\VaR_{h,p}$ is hence
$\ell _{\VaR}=2z_{1-\eta /2}\sqrt{p(1-p)}/(\sqrt{n}f_{x_{h}}(\VaR_{h,p}))$, which is
linked to the precision of the MC estimate.
Setting $\ell _{\VaR} \leq 10^{-a}$ for some integer $a$, this equation can be
solved for $R,$ giving
\begin{equation}
n\geq \left\lceil \frac{4z_{1-\eta /2}^{2}p(1-p)10^{2a}}{f_{x_{h}}^{2}(\VaR%
_{h,p})}\right\rceil .  \label{eq_length_VaR}
\end{equation}

Similarly, consider the MC estimation of $\ES_{h,p}$ for given $\VaR_{h,p}$.
Assuming $\VaR_{h,p}$ known here simplifies derivations without altering the main discussion of MC uncertainty; see \cite{Patton2019} for the joint estimation of $\VaR_{h,p}$ and $\ES_{h,p}$.
The Expected Shortfall could be estimated  by
\begin{equation*}
m_{h,p}=\frac{1}{p n}\sum_{j=1}^{n}-x_{h,r}^{\ast }1(x_{h,j}^{\ast }\leq \VaR_{h,p})=:\frac{1}{pn}\sum_{j=1}^{n}v_{h,j}
\end{equation*}%
with $v_{h,j}:=-x_{h,j}^{\ast }1(x_{h,j}^{\ast }\leq \VaR_{h,p})$.
Observe that this MC estimator is consistent when $\ES_{h,p}$ exists, which is the case thanks to the results in Theorem \ref{theorem_1}.

Let further $V_{h}^{2}:=\E ( v_{h,j}^{2} ) -\E ( v_{h,j}) ^{2}$ where
\begin{equation}\label{eq_mom_v}
\E(v_{h,j}^n)=\int_{-\infty }^{-\VaR_{h,p}}(-u)^n~f_{x_{h}}( u) \mathrm{d}u, \qquad n=1,2.
\end{equation}%
Observe that these expectations exist thanks to the results in Theorem \ref{theorem_1}. Further, note that
$v_{h,r}/p$ has expectation $\ES_{h,p}$ and variance $V_{h}^{2}/p^{2}$, and hence $\E(v_{h,r})=p\ES_{h,p}$.

Application of the central limit theorem, see e.g. \cite{Dudevicz1988} Theorem 6.3.2.,  to $m_{h,p}$ implies that
\begin{equation*}
\sqrt{n}\left( m_{h,p}-\ES_{h,p}\right) \overset{w}{\rightarrow }\rN\left( 0,%
\frac{V_{h}^{2}}{p^{2}}\right) . 
\end{equation*}%
Thus a MC large-$n$ confidence interval for $\ES_{h,p}$ using $m_{h,p}$ at level $\eta $ is given by $m_{h,p}\pm z_{1-\eta /2}V_{h}/(p\sqrt{n})$. The length (precision) of the confidence interval for $\ES_{h,p}$ is hence $%
\ell _{\ES}=2z_{1-\eta /2}V_{h}/(p\sqrt{n}).$ Setting $\ell _{\ES} \leq
10^{-a}$ for some integer $a$, this equation can be solved for $n$, giving
\begin{equation}
n\geq \left\lceil \frac{4z_{1-\eta /2}^{2}V_{h}^{2}10^{2a}}{p^{2}}%
\right\rceil .  \label{eq_lenghtCI_ES}
\end{equation}

Values of $n$ from \eqref{eq_length_VaR}
are reported in Table \ref{table_10} for the selected precision level $a=5$ and $h=2$, using the values of $\alpha$ and $\beta$ from Table \ref{table_0} and with reference to the standardized variate.
In Table \ref{table_10}, $f_{x_{h}}(\VaR_{h,p})$ in \eqref{eq_length_VaR} is computed using the exact formula \eqref{eq_predictive_Gaussian}.

From Table \ref{table_10} one deduces that a large number of replications $n$ is required to compute a confidence interval at level $1-\eta $ for $\VaR_{h,p}$
for given $a$. Note that the values of $n$ are large also because of the factor $f_{x_{h}}^{2}(\VaR_{h,p})$ and
$p^2$ in the denominators of \eqref{eq_length_VaR} and \eqref{eq_lenghtCI_ES}, respectively.

Values of $n$ from
\eqref{eq_lenghtCI_ES}
are reported in Table \ref{table_11} for the selected precision level $a=5$ and $h=2$, using the values of $\alpha$ and $\beta$ from Table \ref{table_0}, and with reference to the standardized variate.
In Table \ref{table_11}, $V_{h}^2$ in \eqref{eq_lenghtCI_ES} is evaluated using numerical integration in \eqref{eq_mom_v} for $n=2$ with $f_{x_{h}}(\cdot)$ computed as in \eqref{eq_length_VaR}.
Also from Table \ref{table_11} one deduces that a large number of replications $R$ is required to compute a confidence interval at level $1-\eta $ for $\ES_{h,p}$.

\begin{table}[tbp] \centering%
\caption{Number of replications $n$ for a MC confidence interval on $\VaR_{2,p}$, see eq. \eqref{eq_lenghtCI_ES}, \newline at given MC coverage level $1-\eta$.}\label{table_11}%
\begin{tabular}{lcccc}
\hline
	&	$p=0.05$	&	0.025	&	0.01	&	0.005	\\
\hline									
$\eta=0.05$	&	5.0484E+12	&	1.2813E+13	&	4.0575E+13	&	9.3989E+13	\\
$\eta=0.01$	&	8.7196E+12	&	2.2131E+13	&	7.0081E+13	&	1.6234E+14	\\
\hline
\end{tabular}%
\end{table}%

More importantly, 
because of the nature of confidence intervals,
there is probability $\eta $ that each of $\VaR_{h,p}$ or $\ES_{h,p}$ does not fall
within its MC confidence interval.
Decreasing $\eta $ does not
offer a solution to this problem, because the quantile $z_{1-\eta /2}$ of
the standard normal distribution would diverge.

One hence concludes that the MC estimation of $\VaR_{h,p}$ or $\ES_{h,p}$ is costly
in terms of number of replications $n$, and it does not guarantee any given
level of numerical precision $a$, because of the probability $\eta $ of $\VaR_{h,p}$ or $\ES_{h,p}$ to fall outside its MC confidence interval. This is in contrast with the ease and
precision of the exact formulae \eqref{eq_predictive_Gaussian} and \eqref{eq_predictive_cdf} provided in this paper.

Similar consideration apply the to direct nonparametric estimation of the \pof\ density.

\section{Uncertainty regions for \pof\ functionals}\label{sec_forecast}
\color{\mycol}This section discusses how uncertainty regions can be constructed for \pof\ functionals to reflect estimation uncertainty, making use of the explicit formulae in the paper.

Let $\vtheta=(\omega, \alpha, \beta, \lambda)'$ indicate the parameters of the GARCH(1,1) in eq. \eqref{eq_GARCH11}, and assume that the model has been estimated on a sample of data $\{x_t\}_{t=-T+1}^0$ by Quasi Maximum Likelihood (QML). Note that the estimation sample includes $T>0$ observations indexed by negative values of $t$. Let $\widehat{\vtheta}$ be the corresponding QML and
$\vtheta _ 0$ the (pseudo)-true values.

Under appropriate regularity conditions,
see
\citet{Lee1994}, \citet{Jensen2004} and \citet{Arvanitis2017} and references therein,
one has results of the type
$T^{\frac{1}{2}}\mR'(\widehat{\vtheta}-\vtheta_0) \overset{w}{\rightarrow } \rN( \vzeros, \mOmega_{\mR})$, where $\overset{w}{\rightarrow } $ indicates convergence in distribution as $T \rightarrow \infty$, and
$\mR$ indicates a full-column-rank matrix with $r$ columns.
This allows to construct asymptotic confidence regions of the type
\begin{equation}\label{eq_CI}
A_\eta=\{ \mR'\vtheta: (\widehat{\vtheta}-\vtheta)'\mR \mOmega_{\mR}^{-1}\mR'(\widehat{\vtheta}-\vtheta)\leq c_\eta\}
\end{equation}
where $\Pr (w\leq c_\eta)=1-\eta$ and $w \sim \chi^2( r)$, and $\mR$ is a full column rank matrix with $r$ columns.

This region has the property that $\Pr (\mR'\vtheta_0 \in A_\eta) \rightarrow 1-\eta$. Note that in \eqref{eq_CI} $\mOmega_{\mR}$ can be replaced by a consistent estimator. A special case of this is when $\mR$ is chosen equal to the identity $\mI$; in this case \eqref{eq_CI} gives the confidence ellipsoid for the unrestricted vector $\vtheta$; this is default case in the following.

Let $\vvarphi$ be a (multivariate) functional of interest, such as the $\VaR$ or the $\ES$, or both, which depends on $\vtheta$, $\vvarphi= \vvarphi (\vtheta)$. Define also the set of values $B_\eta$ taken by the $\vvarphi$ map for any value of $\vtheta$ in $A_\eta$, i.e.
\begin{equation}\label{eq_CI2}
B_\eta =\vvarphi (A_\eta) = \{ \vvarphi (\vtheta), \vtheta \in A_\eta\}.
\end{equation}
Then the following proposition shows that $B_\eta$ is an uncertainty region for $\vvarphi$ with at least asymptotic coverage equal to $1-\eta$.
\begin{prop}[Uncertainty region]
$B_\eta$ is a uncertainty region for $\vvarphi$ with at least asymptotic coverage equal to $1-\eta$, i.e. $\Pr (\vvarphi(\vtheta_0) \in B_\eta) \rightarrow \gamma \geq 1-\eta$.
\end{prop}
\begin{proof}
 See \citet{Fanelli2010} Proposition 1.
\end{proof}

In practice, one needs to compute the set $\vvarphi(A_\eta)$. Assume for simplicity that $\vvarphi(A_\eta)$ is univariate, indicated here as $\varphi(A_\eta)$. An uncertainty interval would be $(\varphi_1, \varphi_2)$ where
$\varphi_1 = \inf_{\vtheta \in A_\eta}\{\varphi(\vtheta)\}$ and
$\varphi_2 = \sup_{\vtheta \in A_\eta}\{\varphi(\vtheta)\}$.

One way to approximate the interval $(\varphi_1, \varphi_2)$ is to calculate the extremes of $\varphi(\theta)$ for a grid of points $\theta$ in $A_\eta$. Let $\mathcal{A}_\eta \subset A_\eta$ be this grid of points; one can then calculate
$(\varphi_1^{\star}, \varphi_2^{\star})$ as an approximation to $(\varphi_1, \varphi_2)$ where
$\varphi_1^{\star} = \min_{\vtheta \in \mathcal{A}_\eta}\{\varphi(\vtheta)\}$ and
$\varphi_2^{\star} = \max_{\vtheta \in \mathcal{A}_\eta}\{\varphi(\vtheta)\}$.
Appendix \ref{sec_grid} illustrates how to construct a grid of points in $A_\eta$.

Two cases were considered for illustration. The first case, labelled `Microsoft stock returns', corresponds to a GARCH (1,1,) estimated on daily log returns of the Microsoft stock price, over the period 2010-12-08 to 2018-11-15, for a total of 2000 observations. The GARCH(1,1) ML estimates were
$\hat{\omega}= 0.048977 (0.0049464)$,
$\hat{\alpha} = 0.078824 (0.0075383)$,
$\hat{\beta}= 0.88389  (0.0092256)$, with estimated standard errors in parenthesis. The estimated asymptotic variance covariance was saved and used to compute the estimation uncertainty region for $\VaR_{2,p}$ and $\ES_{2,p}$.
Table \ref{table_grid} reports the results.

The second case, labelled `One simulation run', corresponds to the simulation of 1000 data points from a GARCH(1,1) with
$\omega = 0.02$
$\alpha = 0.1$ and
$\beta = 0.8$.
The resulting ML estimates were
$\hat{\omega} = 	0.035234	(	0.0128	)$,
$\hat{\alpha} =	0.13336	(	0.0896	)$,
$\hat{\beta} = 	0.68463	(	0.0334	)$,
with standard errors in parenthesis.
The estimated asymptotic variance covariance was saved and used  to compute the estimation uncertainty intervals for $\VaR_{2,p}$ and $\ES_{2,p}$. Table \ref{table_grid} reports the results.

For both cases in Table \ref{table_grid}, 200 points were used in the grid, half of which were selected as image of points $\vtheta$ for which the inequality in \eqref{eq_CI} is valid as an equality, i.e. points on the surface of the confidence ellipsoid. The last column in Table \ref{table_grid} reports how many of the extremes in each row were found corresponding to $\vtheta$ values on the surface. It can be seen that many of these extremes come from points on the surface, but not all.
Increasing the number of points in the grid to 2000 gave marginal improvements for the extremes.\footnote{The extremes varied for less that $4.1 \cdot 10^{-5}$ for the Microsoft case and for less than $3.5 \cdot 10^{-4}$ for the One simulation run. For 2000 points, 4 out of 4 (respectively 3 out of 4) of the extremes came from points on the surface for the Microsoft case (respectively for the One simulation run).}
More details on the computations behind Table \ref{table_grid} are reported in Appendix \ref{sec_grid}.

\begin{table}[tbp] \centering%
\caption{Estimation-uncertainty regions of $\VaR_{2,0.01}$ and $\ES_{2,0.01}$. \newline Grid over 200 points, half of which derived from points on the surface of the estimation confidence ellipsoid.}\label{table_grid}
\begin{tabular}{lcccccccc}
\hline
& \multicolumn{2}{c}{interval for $\VaR_{2,0.01}$} &  &  & \multicolumn{2}{c}{
interval for $\ES_{2,0.01}$} &  &number of extremes  \\ \cline{2-7}
 & min & max &  &  & min & max & &derived from surface points\\ \hline
Microsoft stock returns & $-2.3383$ & $-2.3298$ &  &  & 2.6733 & 2.6957 && 4
out of 4 \\
One simulation run & $-2.4197$ & $-2.3263$ &  &  & 2.6652 & 2.9144 && 2 out
of 4 \\ \hline
\end{tabular}%
\end{table}%

One could ask whether analogues to this procedure exist which use MC in place of the exact formulae, where each map $\vvarphi(\theta)$ is replaced by MC simulation plus MC estimation of $\vvarphi(\cdot)$. The MC approach implies a large computational burden, because of the added MC simulation and estimation burden associated with the estimation of $\vvarphi(\cdot)$ map. Moreover, the inherent limitations associated with MC confidence
interval discusses in Section \ref{sec_appl} would apply here, which would add extra uncertainty for the estimation of $\vvarphi(\cdot)$. 
This additional layer of MC uncertainty is completely avoided by the present exact methods.

In other words, the uncertainty regions produced via the present exact methods only reflect in-sample estimation uncertainty associated with the GARCH parameter, but not the MC simulation and estimation uncertainty of the $\vvarphi(\cdot)$ map.


\color{black}

\section{Conclusions}
\label{sec_conc}
This paper presents the analytical form of the \pof\ density of a
GARCH(1,1) process. This can be used to evaluate the probability of tail events
or of quantities that may be of interest for value at risk calculations.
The exact formulae improve on approximation methods based on moments, or on Monte Carlo simulation and estimation.

The exact formuale show that, while the \pof\ density can be very far from normal, for common parameter values often encountered in applications, the discrepancy of the \pof\ density from the Gaussian distribution \PP{can be} small. These results could not be obtained without the explicit form of the \pof\ density.

\PP{The present exact results are shown to imply easy-to-compute uncertainty regions for risk functionals, so as to reflect estimation uncertainty. These tools are not available for alternatives based on approximations or MC simulations and estimation of functionals.}

The techniques in this paper can be \PP{extended to the case} of symmetric innovations density $g(\cdot)$ different from the N(0,1) one. Different densities imply distinct subsequent (negative) binomial expansions of expression (\ref{eq_sigma2_h}) for $\sigma _{t}^{2}$, and different auxiliary convergence conditions on the GARCH coefficients, similarly to Assumption \ref{ass_1}. These extensions are left to future research.




\bibliography{ects}
\bibliographystyle{Chicago_pp}

\appendix
\section{Proofs}
\label{sec_app_proofs}
The proofs of the Theorems are based on several Lemmas, which are reported first.

\begin{lemma}[Limits of $\Psi$]\label{lemma-Psi}
$\Psi (a,c;z)\rightarrow 0$ for $c\rightarrow -\infty $ for real and
positive $a$ and $z$ and \PP{$\Psi (a,c;z)\rightarrow 0$ for $a\rightarrow
\infty $ for real and positive $c$ and $z$}.
\end{lemma}

\begin{proof}

The proof uses the Lebesgue dominated convergence theorem,
see e.g. Theorem 10.27 in \cite{Apostol1981}.
Consider the integral representation \eqref{eq:tricomi} of $\Psi (a,c;z)$
for real and positive $a$ and $z$.
Note that for negative $c$ and $t\geq 0$ one has
\begin{equation*}
k_{n}(t):=\PP{\frac{1}{\Gamma (a)}}e^{-zt}t^{a-1}\left( 1+t\right) ^{c-a-1}\leq
\frac{1}{\Gamma (a)}e^{-zt}t^{a-1}=:r(t),
\end{equation*}%
where $k_{n}(t)$, $r(t)>0$, $\int_{\mathbb{R}_{+}}k_{n}(t)\mathrm{d}t=\Psi
(a,c;z)$ and $\int_{\mathbb{R}_{+}}r(t)\mathrm{d}t=\frac{1}{\Gamma (a)}\int_{%
\mathbb{R}_{+}}e^{-zt}t^{a-1}\mathrm{d}t=\frac{1}{\Gamma (a)}z^{-a}\Gamma
(a)=z^{-a}$;\ this shows that $k_{n}(t)$ is dominated by the function $r(t)$%
, which is Lebesgue-integrable on $\mathbb{R}_{+}$. The notation $k_{n}(t)$
is chosen here to indicate that a sequence of values $a_{n}$ or $c_{n}$ will
be constructed.

Next observe that for any $t>0$, and for $c_{n}\rightarrow -\infty $, one
has $k_{n}(t):=e^{-zt}t^{a-1}\left( 1+t\right) ^{c_{n}-a-1}\rightarrow 0$.
Hence $k_{n}(t)$ converges to the zero function $k(t):=0$ on the whole $%
\mathbb{R}_{+}$, except for the point $t=0$. By the dominated convergence
theorem, $\lim_{c_{n}\rightarrow -\infty }\Psi
(a,c_n;z)=\lim_{c_{n}\rightarrow -\infty }\int_{\mathbb{R}_{+}}k_{n}(t)\mathrm{%
d}t=\int_{\mathbb{R}_{+}}k(t)\mathrm{d}t=0$. This proves that $\Psi
(a,c;z)\rightarrow 0$ for $c\rightarrow -\infty $ for real and positive $a$
and $z$.

\PP{
Let now $a_{n}\rightarrow \infty $ and observe that for any $t>0$, $\frac{t}{%
1+t}<1$ and $\Gamma (a_n)\rightarrow \infty $, and hence
\begin{equation*}
k_{n}(t):=\frac{1}{\Gamma (a_n)}e^{-zt}t^{a_n-1}\left( 1+t\right) ^{c-a_n-1}=\frac{1}{%
\Gamma (a_n)}e^{-zt}(1+t)^{c-2}\left( \frac{t}{1+t}\right) ^{a_n-1}\rightarrow 0.
\end{equation*}%
Hence $k_{n}(t)$ converges to the zero function $k(t):=0$ on the whole $%
\mathbb{R}_{+}$. By the dominated convergence theorem, $\lim_{a_{n}%
\rightarrow -\infty }\Psi (a_n,c;z)=\lim_{a_{n}\rightarrow -\infty }\int_{%
\mathbb{R}_{+}}k_{n}(t)\mathrm{d}t=\int_{\mathbb{R}_{+}}k(t)\mathrm{d}t=0$.
This proves that $\Psi (a,c;z)\rightarrow 0$ for $a\rightarrow \infty $ for
real and positive $c$ and $z$.}
\end{proof}

\begin{proof}[Proof of Lemma \ref{lemma01}]
Consider the
transformation theorem for $z_{h}=x_{h}^{2}$;\ from standard results, see
e.g. \citet{Mood1974}, page 201, Example 19, one has
\begin{equation}\label{eq_trans2}
f_{z_{h}}(w_{h})=\left( \frac{1}{2}\frac{1}{\sqrt{w_{h}}}f_{x_{h}}(-\sqrt{%
w_{h}})+\frac{1}{2}\frac{1}{\sqrt{w_{h}}}f_{x_{h}}(\sqrt{w_{h}})\right)
1_{w_{h}\geq 0}.
\end{equation}%
where $1_{A}$ is the indicator function of the event $A$.
Because, by symmetry, one has $%
f_{x_{h}}(-\sqrt{w_{h}})=f_{x_{h}}(\sqrt{w_{h}})$, \PP{\eqref{eq_trans2}}
simplifies to
$
f_{z_{h}}(w_{h})=w_{h}^{-\frac{1}{2}}f_{x_{h}}(\sqrt{w_{h}})1_{(w_{h}\geq 0)},
$
or, letting $u_{h}$ indicate $w_{h}^{\frac{1}{2}}$, and solving for $%
f_{x_{h}}(u_{h})$, one finds $f_{x_{h}}(u_{h})=\left\vert u_{h}\right\vert
f_{z_{h}}(u_{h}^{2})$, which is \eqref{eq_predictive_Gauss_xt}. Note that
the expression with the absolute value is also valid for $u_{h}=-\sqrt{w_{h}}
$. This proves \eqref{eq_predictive_Gauss_xt}.

\PP{One has by assumption that $f_{\varepsilon }(\epsilon ):=g(\epsilon
^{2}):=(2\pi )^{-\frac{1}{2}}\exp (-\epsilon ^{2}/2)$. Hence, simple
applications of the transformation theorem cited above imply $%
f_{x_{t}|x_{1},\dots ,x_{t-1}}(u|u_{1},\dots u_{t-1})=(2\pi \sigma
_{t}^{2})^{-\frac{1}{2}}\exp \left( -\frac{1}{2}\frac{u^{2}}{\sigma _{t}^{2}}%
\right) =(\sigma _{t}^{2})^{-\frac{1}{2}}g\left( \frac{u^{2}}{\sigma _{t}^{2}%
}\right) $ and $f_{z_{t}|x_{1},\dots ,x_{t-1}}(w|u_{1},\dots u_{t-1})=(w)^{-%
\frac{1}{2}}f_{x_{t}|x_{1},\dots ,x_{t-1}}(w^{\frac{1}{2}}|u_{1},\dots
u_{t-1})1_{w\geq 0}=(w\sigma _{t}^{2})^{-\frac{1}{2}}g\left( \frac{w}{\sigma
_{t}^{2}}\right) $, from which Eq. \eqref{eq_decompose} follows.} 
\end{proof}

\begin{proof}[Proof of Lemma \ref{Lemma_2}]
Consider $f_{\vz,z_{h} |\vvarsigma} (\vw,w_{h}|\vs)$ from \eqref{eq_decompose}, and consider the transformation of from $\vz$ to $\vy$. Observe that the domain of integration
remains $\mathbb{R}_{+}^{h-1}$, that the inverse transformation is $%
z_{t}=\beta \sigma _{t}^{2}y_{t}/\alpha _{t}$, with Jacobian $%
\gamma _{h}\prod_{t=1}^{h-1}\sigma _{t}^{2}$, where $\gamma _{h}:=\beta
^{h-1}/(\prod_{t=1}^{h-1}\alpha _{t})$. Hence one finds
\begin{equation*}
f_{\vy, z_h|\vvarsigma}\left( \vv, w_h|\vs\right) =\gamma _{h}^{\frac{1}{2}%
}\prod_{t=1}^{h-1}\left(
v_{t}^{-\frac{1}{2}}
g\left( \frac{\beta }{\alpha
_{t}}v_{t}\right) \right) \cdot \left( w_{h}\sigma _{h}^{2}\right) ^{-\frac{1%
}{2}}g\left( \frac{w_{h}}{\sigma _{h}^{2}}\right)
\end{equation*}%
from which \eqref{eq_last_I} follows, as in \eqref{eq_integral_I}.
\end{proof}

%


\begin{lemma}[Conditions on $\protect\beta $]
\label{Lemma_beta_cond} Assumption $\ref{ass_1}$ ensures that for any $j\geq 2$
\begin{equation}
\omega \left( 1-\sum_{i=1}^{j-1}\beta ^{i}\right) \leq \beta ^{j}\sigma
_{1}^{2}, \label{eq_inductive0}
\end{equation}
\PP{which implies that in \eqref{eq_vola} one has
\begin{equation}\label{eq_ineq0}\omega \leq \left( 1+v_{h-1}\right) \beta \sigma _{h-1}^{2}.
\end{equation}}
\end{lemma}
\begin{proof} For $j=2$ the inequality (\ref{eq_inductive0}%
) reads $\beta ^{2}\sigma _{1}^{2}+\omega \beta -\omega \geq 0$.
Solving the quadratic on the l.h.s. for $\beta $ one finds two roots, $\beta
_{1}=(-\omega -\sqrt{\omega ^{2}+4\omega \sigma _{1}^{2}})/(2\sigma
_{1}^{2})<0$ and $\myb=(-\omega +\sqrt{\omega ^{2}+4\omega \sigma
_{1}^{2}})/(2\sigma _{1}^{2})>0$, so that the quadratic is non-negative for $%
\beta \leq \beta _{1}$ or for $\beta >\myb$. Because $\beta _{1}<0$ is not possible, this
holds only when $\beta \geq \myb$. This proves that (\ref%
{eq_inductive0}) is valid for $j=2$ for $\beta \geq \myb$ and a
fortiori also for $\beta \geq \max \{\frac{1}{2},\myb\}$.

An induction approach is used for $j>2$. Assume that (\ref{eq_inductive0})
is valid for some $j=j_{0}\geq 2$ and $\beta \geq \max \{\frac{1}{2},\myb \}$; it can then be shown that (\ref{eq_inductive0}) is valid also
replacing $j$ with $j+1$. To see this, take (\ref{eq_inductive0}) for $%
j=j_{0}$ and multiply by $\beta $. One finds
\begin{equation*}
\omega \left( \beta -\sum_{i=1}^{j_{0}-1}\beta ^{i+1}\right) \leq \beta
^{j_{0}+1}\sigma _{1}^{2}.
\end{equation*}%
Because $\beta \geq \frac{1}{2}$, one has $\omega (1-\beta )\leq \omega
\beta $, so that,
\begin{equation*}
\omega \left( 1-\beta -\sum_{i=1}^{j_{0}-1}\beta ^{i+1}\right) \leq \omega
\left( \beta -\sum_{i=1}^{j_{0}-1}\beta ^{i+1}\right) \leq \beta
^{j_{0}+1}\sigma _{1}^{2}.
\end{equation*}%
Rearranging $1-\beta -\sum_{i=1}^{j_{0}-1}\beta ^{i+1}$ as $%
1-\sum_{i=1}^{j_{0}}\beta ^{i}$, one finds that
\eqref{eq_inductive0}
holds
also for $j=j_{0}+1$. The induction step hence proves that (\ref%
{eq_inductive0}) holds for any $j$ if $\beta \geq \max \{\frac{1}{2},\myb\}$.

\PP{To show \eqref{eq_ineq0}, observe that the minimum value for $\beta \sigma _{h-1}^{2}$
corresponds to $v_{h-2}=\dots =v_{1}=0$, which equals $\omega
\sum_{i=1}^{j}\beta ^{i}+\beta ^{h-1}\sigma _{1}^{2}$. The last expression is
greater than $\omega $ by \eqref{eq_inductive0}, and hence $\omega \leq \beta \sigma _{h-1}^{2}\leq
\left( 1+v_{h-1}\right) \beta \sigma _{h-1}^{2}$. }
\end{proof}

\PP{\begin{lemma}[Binomial expansion]\label{Lemma_exp}
Under assumption $\ref{ass_1}$, the following expansion holds for any $r$
\begin{align}\nonumber
& (\sigma _{h}^{2})^{r} =\sum_{k_{1}=0}^{r}\sum_{k_{2}=0}^{r-k_{1}}\cdots
\sum_{k_{h-2}=0}^{r-K_{h-3}}\binom{r}{k_{1}}\binom{%
r-k_{1}}{k_{2}}\cdots \binom{r-K_{h-3}}{k_{h-2}}\omega ^{K_{h-2}}\beta
^{(h-2)r-U_{h-2}}\cdot \\
& \qquad\qquad\qquad\qquad \cdot
\prod_{t=1}^{h-2}\left( 1+v_{h-t}\right)
^{r-K_{t}}(\omega +(1+v_{1})\beta \sigma _{1}^{2})^{r-K_{h-2}}.\label{eq_sig2r_expand}
\end{align}
where $K_t:=\sum_{i=1}^{t}k_i$, $U_{j}:=\sum_{i=1}^{j}K_{i}=\sum_{i=1}^{j} (j-i+1) k_{i}$, and
the sums
$\sum_{k_{1}=0}^{r}\sum_{k_{2}=0}^{r-k_{1}}\cdots
\sum_{k_{h-2}=0}^{r-K_{h-3}}$ extend to $\infty$ if $r\notin \mathbb{N}$.
\end{lemma}
\begin{proof}
Under Assumption $\ref{ass_1}$, Lemma \ref{Lemma_beta_cond} implies that one can employ binomial expansions of $(\sigma
_{h}^{2})^{r}$ where $\sigma _{h}^{2}=\omega +\left( 1+v_{h-1}\right) \beta
\sigma _{h-1}^{2}$ using increasing powers of $\omega $ and decreasing
powers of $\left( 1+v_{h-1}\right) \beta \sigma _{h-1}^{2}$. Hence, setting $%
U_{j}:=\sum_{i=1}^{j}(r-K_{i})$,
\begin{align*}
& (\sigma _{h}^{2})^{r}  = \sum_{k_{1}=0}^{r}\binom{r}{k_{1}}\omega ^{k_{1}}\left(
1+v_{h-1}\right) ^{r-k_{1}}\beta ^{r-k_{1}}\left( \sigma _{h-1}^{2}\right)
^{r-k_{1}}  = \\
& =\sum_{k_{1}=0}^{r}\sum_{k_{2}=0}^{r-k_{1}}\binom{r}{k_{1}}\binom{r-k_{1}}{%
k_{2}}\omega ^{k_{1}+k_{2}}\left( 1+v_{h-1}\right) ^{r-k_{1}}\beta
^{r-k_{1}+r-k_{1}-k_{2}}\left( 1+v_{h-2}\right) ^{r-k_{1}-k_{2}}\left(
\sigma _{h-2}^{2}\right) ^{r-k_{1}-k_{2}} \\
& =\dots  \\
& =\sum_{k_{1}=0}^{r}\sum_{k_{2}=0}^{r-k_{1}}\cdots
\sum_{k_{h-2}=0}^{r-K_{h-3}}\binom{r}{k_{1}}\binom{r-k_{1}}{k_{2}}\cdots
\binom{r-K_{h-3}}{k_{h-2}}\omega ^{K_{h-2}}\beta ^{(h-2)r-U_{h-2}}\cdot \\
& \qquad\qquad\qquad\qquad\qquad \cdot \prod_{t=1}^{h-2}\left( 1+v_{h-t}\right)
^{r-K_{t}}\left( \omega +(1+v_{1})\beta \sigma _{1}^{2}\right) ^{r-K_{h-2}}
\end{align*}
This proves the claim.
\end{proof}
}
\PP{
\begin{lemma}[Integrals]\label{Lemma_psi_int}
One has
\begin{align}\label{eq_psi_int}
& H_{n}(c,z):=\int_{\mathbb{R}_{+}}\exp \left( -zv\right) \left( 1+(1+v)c\right) ^{n }
v^{-\frac{1}{2}}\mathrm{d}v=\left( c+1\right) ^{n +\frac{1}{2}}c^{-\frac{1%
}{2}}\pi ^{\frac{1}{2}}\Psi \left( \frac{1}{2},n +\frac{3}{2};\frac{c+1}{c%
}z\right) \\
& \myG_{n}(z):=\int_{0}^{\infty}\exp  (-z v) (1+v)^n v^{-\frac{1}{2}} \mathrm{d}v =
\pi^{\frac{1}{2}} \Psi \left(\frac{1}{2}, n+\frac{3}{2};z\right) \label{eq:Bt}
\end{align}
\end{lemma}
}
\PP{
\begin{proof}
Set $b:=c+1$ and $t:=mv\ $with $m:=c/(c+1)$ so that $1+c+cv\ =b(1+t)$. Note
that $m^{-1}\mathrm{d}t=\mathrm{d}v$, so that
\begin{align*}
& \int_{\mathbb{R}_{+}}\left( 1+(1+v)c\right) ^{n }\exp \left( -zv\right)
v^{-\frac{1}{2}}\mathrm{d}v =b^{n }\int_{\mathbb{R}_{+}}(1+t)^{n
}\exp \left( -\frac{z}{m}t\right) \left( m^{-1}t\right) ^{-\frac{1}{2}}m^{-1}%
\mathrm{d}t \\
&=b^{n }m^{-\frac{1}{2}}\int_{\mathbb{R}_{+}}(1+t)^{n }\exp \left( -%
\frac{z}{m}t\right) t^{-\frac{1}{2}}\mathrm{d}t=\left( c+1\right) ^{n +%
\frac{1}{2}}c^{-\frac{1}{2}}\pi ^{\frac{1}{2}}\Psi \left( \frac{1}{2},n +%
\frac{3}{2};\frac{z}{m}\right) ,
\end{align*}%
see \eqref{eq:tricomi}. The case of \eqref{eq:Bt} is obtained as the set of the last 2 equalities setting $m=b=1$.
\end{proof}
}

\PP{
\begin{lemma}[Coefficients $A_h(\cdot)$]
\label{Lemma_cj1} Let
\begin{equation}
A_{h}(r):= (\omega + \beta \sigma_1^2 )^{-r-\frac{1}{2}} \int_{\mathbb{R}_{+}^{h-1}}\exp \left( -\frac{1}{2}%
\sum_{t=1}^{h-1}\frac{\beta }{\alpha _{t}}v_{t}\right) (\sigma _{h}^{2})^{r}\prod_{t=1}^{h-1}\frac{\mathrm{d}v_{t}}{\sqrt{v_{t}}};
\label{eq_int_repr_cj}
\end{equation}%
then
\begin{equation}\label{eq_A2r}
A_{2}(r) =\left( \beta \sigma _{1}^{2}\right) ^{-\frac{1}{2}}
\pi ^\frac{1}{2}\Psi \left( \frac{1}{2},r+\frac{3}{2};\frac{\omega+\beta\sigma_1^{2}}{2\alpha_1\sigma_1^{2}}\right),
\end{equation}
and assuming that $\eqref{eq_inductive0}$ holds for $2\leq j\leq
h$, one has for $h\geq3$
\begin{align}\nonumber
A_{h}(r) &=\left( \beta \sigma _{1}^{2}\right) ^{-\frac{1}{2}}\pi ^\frac{h-1}{2}  \sum_{k_{1},\dots,k_{h-2}}\binom{r}{k_{1}}\binom{r-k_{1}}{k_{2}}\cdots
\binom{r-K_{h-3}}{k_{h-2}}\omega ^{K_{h-2}}\beta ^{(h-2)r-U_{h-2}} \cdot \\ &  \cdot \prod_{t=1}^{h-2}
\Psi \left( \frac{1}{2},r-K_{t}+\frac{3}{2};\frac{\beta}{2\alpha_{h-t}}\right)
\Psi \left( \frac{1}{2},r-K_{h-2}+\frac{3}{2};\frac{\omega+\beta\sigma_1^{2}}{2\alpha_1\sigma_1^{2}}\right),
\label{eq_long}
\end{align}
where $K_0:=0$, $K_i:=\sum_{t=1}^{i}k_i$, $U_{j}:=\sum_{i=1}^{j}K_{i}=\sum_{i=1}^{j} (j-i+1) k_{i}$,
$\sum_{k_{1},\dots,k_{h-2}}:=\sum_{k_{1}=0}^{r}\sum_{k_{2}=0}^{r-k_{1}}\cdots
\sum_{k_{h-2}=0}^{r-K_{h-3}}$, and the sums extend to $\infty$ if $r\notin \mathbb{N}$. Note that \eqref{eq_long} reduces to
\eqref{eq_A2r} for $h=2$, because the sum $\sum_{k_{1},\dots,k_{h-2}}$ and the product $\prod_{t=1}^{h-2}$ are empty and $K_{h-2}=K_{0}=0$.
\end{lemma}
}
\PP{
\begin{proof}
Set $h=2$ in \eqref{eq_int_repr_cj} and note that
\begin{equation*}
\int_{\mathbb{R}_{+}}\exp \left( -\frac{\beta }{2\alpha _{1}}
v_{1}\right) \left( \omega +\left( 1+v_{1}\right)
\beta \sigma _{1}^{2}\right) ^{r}v_{1}^{-\frac{1}{2}}\mathrm{d}v_{1}
= \omega ^{r} H_{r}\left(\beta \sigma_1^2 \omega^{-1},\frac{\beta }{2\alpha _{1}}\right)
\end{equation*}
so that by \eqref{eq_psi_int} eq. \eqref{eq_A2r} holds.
Next consider the case $h\geq3$. Under $\eqref{eq_inductive0}$
one can use expansion \eqref{eq_sig2r_expand} in
\eqref{eq_int_repr_cj}. Integrating one finds
\begin{align*}
A_{h}(r) &=\sum_{k_{1},\dots,k_{h-2}}\binom{r}{k_{1}}\binom{r-k_{1}}{k_{2}}\cdots
\binom{r-K_{h-3}}{k_{h-2}}\omega ^{K_{h-2}}\beta ^{(h-2)r-U_{h-2}} \cdot \\ \nonumber &  \qquad \qquad \cdot \prod_{t=1}^{h-2}\myG_{r-K_{t}}\left(\frac{\beta}{2\alpha_{h-t}}\right)\cdot A_{2}(r-K_{h-2}).
\end{align*}
Using \eqref{eq:Bt} and \eqref{eq_A2r}, one finds \eqref{eq_long}.
\end{proof}
}


\begin{proof}[Proof of Theorem $\protect\ref{theorem_1}$]
The integral to be solved is
\begin{equation}
f_{z_{h}|\vvarsigma}(w_{h}|\vs)= \frac{\sqrt{\gamma _{h}/w_{h}}}{\left( 2\pi
\right) ^{\frac{h}{2}}}\int_{\mathbb{R}_{+}^{h-1}}\exp \left( -\frac{1}{2}%
\left( \sum_{t=1}^{h-1}\frac{\beta }{\alpha _{t}}v_{t}+\frac{w_{h}}{\sigma
_{h}^{2}}\right) \right) (\sigma _{h}^2)^{-\frac{1}{2}}\prod_{t=1}^{h-1}\frac{\mathrm{d}%
v_{t}}{\sqrt{v_{t}}}. \label{eq_fstar}
\end{equation}%
Expand $\exp (-w_{h}/(2\sigma _{h}^{2}))=\sum_{j=0}^{\infty }\frac{\left(
-w_{h}/2\right) ^{j}}{j!}\left( \sigma _{h}^{2}\right) ^{-j}$ and note that
\begin{equation*}
f_{z_{h}|\vvarsigma}(w_{h}|\vs)= \frac{\gamma _{h}^{\frac{1}{2}} w_{h}^{-\frac{1}{2}}}{\left( 2\pi \right) ^{%
\frac{h}{2}}}
\sum_{j=0}^{\infty }
\frac{1}{j!} \left( -\frac{w_{h}}{\omega + \beta \sigma_1^2 }\right) ^{j}
A_h\left(-\frac{1}{2}-j\right),
\label{eq:conditional_density}
\end{equation*}%
\PP{where $A_h(\cdot)$ is defined in \eqref{eq_int_repr_cj}. By \eqref{eq_long} one has
$f_{z_{h}|\vvarsigma}(w_{h}|\vs)= w_{h}^{-\frac{1}{2}}\left( 2\pi \right) ^{-\frac{h}{2}}
\sum_{j=0}^{\infty }
\frac{1}{j!} (-\rho_w)^{j}
c_{j,\vs}$
where $\rho_w := w_{h}/ (\omega + \beta \sigma_1^2 ) $ and
$c_{j,\vs}=\gamma _{h}^{\frac{1}{2}}A_h(-\frac{1}{2}-j)$.}
Marginalizing with respect to $\vvarsigma$,
being all elements in $\calS$ equally likely, one finds from \eqref{eq_fz_overall}
\begin{align*}
f_{z_{h}}(w_{h}) &=2^{-h+1}\sum_{\vs\in \calS}f_{z_{h}|\vvarsigma}(w_{h}|\vs%
)=2^{-h+1}\sum_{\vs\in \calS}\sum_{j=0}^{\infty }
\frac{w_{h}^{-\frac{1}{2}}}{( 2\pi )
^{\frac{h}{2}}}
\left(\frac{-w_{h}}{2(\omega+\beta\sigma_1^2)}\right)^{-j}
\frac{c_{j,\vs}}{j!}\\
&=\frac{w_{h}^{-\frac{1}{2}}}{(2\pi ) ^{\frac{h}{2}}}%
\sum_{j=0}^{\infty }
\left(\frac{-w_{h}}{2(\omega+\beta\sigma_1^2)}\right)^{-j}
\frac{c_{j}}{j!}
\end{align*}
where $c_{j}:=2^{-h+1}\sum_{\vs\in \calS}c_{j,\vs}$. Note that if all $c_{j,\vs}$ do not vary with $\vs$, one has
$c_{j}= c_{j,\vs}$.

In order to show that $f_{z_{h}}(w_{h})$ and $f_{x_{h}}(u_{h})$ are
absolutely summable for finite $w_{h}$, $u_{h}$ consider for instance the case $h=2$
for $f_{z_{2}}(w)$. One has
\begin{equation*}
f_{z_{2}|\varsigma _{1}}(w|s_{1})=\frac{1}{2}(\pi \alpha _{1}\sigma
_{1}^{2})^{-\frac{1}{2}}w^{-\frac{1}{2}}\sum_{j=0}^{\infty }\frac{1}{j!}%
\left( -\rho_w \right) ^{j}\Psi \left( \frac{1}{2};1-j;\xi \right) .
\end{equation*}%
where $\rho_w :=\frac{w}{2\left( \omega +\beta \sigma _{1}^{2}\right) }$, $\xi
=\frac{\omega +\beta \sigma _{1}^{2}}{2\alpha _{1}\sigma _{1}^{2}}>0$.
Because $\Psi \left( \frac{1}{2};1-j;\xi \right) $ is non-negative
\PP{and tends to  0 by Lemma \ref{lemma-Psi} for increasing $j$,}
one has $\sup_{-j\in \mathbb{N}}\Psi \left( \frac{1}{2};1-j;\xi \right)
=M<\infty $, so that
\begin{equation*}
\sum_{j=0}^{\infty }\left\vert \frac{1}{j!}\left( -\rho \right) ^{j}\Psi
\left( \frac{1}{2};1-j;\xi \right) \right\vert \leq \sum_{j=0}^{\infty }%
\frac{1}{j!}\rho ^{j}\Psi \left( \frac{1}{2};1-j;\xi \right) \leq
M\sum_{j=0}^{\infty }\frac{1}{j!}\rho ^{j}=M\exp \left( \rho \right)
\end{equation*}%
where $\exp \left( \rho \right) $ is finite for any finite $\rho_w$.
One hence concludes that the series is absolutely convergent for any finite evaluation point $w$.
The same argument applies to $f_{x_{h}}(u_{h})$ for the case $h=2$. The
case for $h>2$ is similar.
\end{proof}


\begin{proof}[Proof of Corollary \protect\ref{Corollary_2}]
The c.d.f.s are found by integrating termwise the p.d.f from 0 to $w_{h}$ or
$u_{h}$ for positive $u_{h}$. Termwise integration is guaranteed by Theorem
10.26 in \cite{Apostol1981}. This delivers the expressions in \eqref{eq_predictive_cdf} for $w_{h}$ and $u_{h}$ for positive $u_{h}$. The symmetry of $f_{x_{h}}(\cdot )$ implies $F_{x_{h}}(0)=\frac{1}{2}$ and $%
F_{x_{h}}(u_{h})=1-F_{x_{h}}(-u_{h})$. Hence for  $0>u_{h}=-a$, say, with $a>0
$,\ one has%
\begin{align*}
F_{x_{h}}(u_{h}) &=1-F_{x_{h}}(a)=\frac{1}{2}-(2\pi )^{-%
\frac{h}{2}}\sum_{j=0}^{\infty }
\frac{1}{j!\left(
2j+1\right) }a^{2j+1}\PP{(-2(\omega+\beta \sigma_1^2))^{-j}}c_{j} \\
&=\frac{1}{2}+(2\pi )^{-\frac{h}{2}}\sum_{j=0}^{\infty }\frac{1}{j!\left( 2j+1\right) }\left( -a\right) ^{2j+1}\PP{(-2(\omega+\beta \sigma_1^2))^{-j}} c_{j}
\end{align*}%
which proves that the expressions in \eqref{eq_predictive_cdf} in valid also
for negative $u_{h}$.


The moments are derived as follows. From \eqref{eq_fstar} one sees
that
\begin{equation*}
\E(z_{h}^{m}|\vvarsigma)= \frac{\sqrt{\gamma _{h}}}{%
\left( 2\pi \right) ^{\frac{h}{2}}}\int_{\mathbb{R}_{+}^{h}}w_{h}^{m-\frac{1%
}{2}}\exp \left( -\frac{1}{2}\frac{w_{h}}{\sigma _{h}^{2}}\right) \mathrm{d}%
w_{h}\prod_{t=1}^{h-1}\exp \left( -\frac{1}{2}\left( \frac{\beta }{\alpha
_{t}}v_{t}\right) \right) \sigma _{h}^{-1}v_{t}^{-\frac{1}{2}}\mathrm{d}%
v_{t}.
\end{equation*}%
Recall that
$\int_{\mathbb{R}_{+}}\exp \left( -\frac{w}{2\sigma _{h}^{2}}\right) w^{m-%
\frac{1}{2}}\mathrm{d}w=\left( 2\sigma _{h}^{2}\right) ^{m+\frac{1}{2}%
}\Gamma \left( m+\frac{1}{2}\right)$
so that
\begin{align*}
\E(z_{h}^{m}|\vvarsigma) & =
2^{m-\frac{h-1}{2}}\pi ^{-\frac{h%
}{2}}\gamma _{h}^{\frac{1}{2}}\Gamma \left( m+\frac{1}{2}\right) \int_{%
\mathbb{R}_{+}^{h-1}}\prod_{t=1}^{h-1}\exp \left( -\frac{1}{2}\left( \frac{%
\beta }{\alpha _{t}}v_{t}\right) \right) \left( \sigma _{h}^{2}\right)
^{m}v_{t}^{-\frac{1}{2}}\mathrm{d}v_{t} \\
& =2^{m-\frac{h-1}{2}}\pi ^{-\frac{h%
}{2}}\Gamma \left( m+\frac{1}{2}\right)
(\omega+\beta \sigma_1^2)^{m + \frac{1}{2}}
\gamma _{h}^{\frac{1}{2}} A_h (m)
\end{align*}%
where $A_h (\cdot)$ is defined in \eqref{eq_int_repr_cj}, which equals \eqref{eq_long} (or \eqref{eq_long_before}).
Hence
\begin{equation*}
\E(z_{h}^{m})=
2^{m-\frac{3}{2}(h-1)}\pi ^{-\frac{h%
}{2}}\Gamma \left( m+\frac{1}{2}\right)
(\omega+\beta \sigma_1^2)^{m + \frac{1}{2}}
\sum_{\vs\in %
\calS}
\gamma _{h}^{\frac{1}{2}} A_h (m)
\end{equation*}
\end{proof}

\begin{proof}[Proof of Theorem \protect\ref{theorem_2}]
\PP{Consider first the case $h=2$ and
\begin{equation}
f_{z_{2}|\vvarsigma}(w_{2}|\vs)=\frac{\gamma _{2}^{\frac{1}{2}}w_{2}^{-\frac{%
1}{2}}}{2\pi }\int_{\mathbb{R}_{+}}\exp \left( -\frac{1}{2}\left( \frac{%
\beta }{\alpha _{1}}v_{1}+\frac{w_{2}}{\sigma _{2}^{2}}\right) \right)
(\sigma _{2}^{2})^{-\frac{1}{2}}\frac{\mathrm{d}v_{1}}{\sqrt{v_{1}}}.
\label{eq_fstar2}
\end{equation}%
As in the proof of Lemma \ref{Lemma_psi_int}, observe that $\sigma _{2}^{2}$ can be written as $%
b(1+s)$ with $b:=\omega +\beta \sigma _{1}^{2}$ and $s:=v_{1}/q$ where $%
q=(\omega +\beta \sigma _{1}^{2})/\beta \sigma _{1}^{2}$, $\sigma
_{2}^{2}=\omega +(1+v_{1})\beta \sigma _{1}^{2}=\omega +\beta \sigma
_{1}^{2}+\beta \sigma _{1}^{2}v_{1}=b(1+s)$. Next define $\rho
:=w_{2}/(2(\omega +\beta \sigma _{1}^{2}))$, and note that $w_{2}/(2\sigma
_{2}^{2})=\rho /(1+s)$;\ observe also that  $\exp \left( -w_{2}/(2\sigma
_{2}^{2})\right) =\exp \left( -\rho /(1+s)\right) =\exp \left( -\rho \right)
\exp \left( \rho s/(1+s)\right) $, where the last term can be expanded as $%
\exp \left( \rho s/(1+s)\right) =\sum_{j=0}^{\infty }\frac{\rho ^{j}}{j!}%
s^{j}(1+s)^{-j}$.
}
\PP{
Substituting these expression in \eqref{eq_fstar2}, using $\gamma _{2}=\beta
/\alpha _{1}$, $q\mathrm{d}s=\mathrm{d}v_{1}$, and setting $z:=$ $\left(
\omega +\beta \sigma _{1}^{2}\right) /\left( 2\alpha _{1}\sigma
_{1}^{2}\right) =\beta q/(2\alpha _{1})$, one finds%
\begin{eqnarray*}
f_{z_{2}|\vvarsigma}(w_{2}|\vs) &=&\left( 2\pi \right) ^{-1}\gamma _{2}^{%
\frac{1}{2}}w_{2}^{-\frac{1}{2}}\e^{-\rho }\sum_{j=0}^{\infty }\frac{\rho
^{j}}{j!}\int_{0}^{\infty }\exp \left( -\frac{\beta q}{2\alpha _{1}}s\right)
s^{j}(1+s)^{-j}(b(1+s))^{-\frac{1}{2}}(qs)^{-\frac{1}{2}}q\mathrm{d}s \\
&=&\left( 2\pi \right) ^{-1}\left( \frac{\beta }{\alpha _{1}}\right) ^{\frac{%
1}{2}}w_{2}^{-\frac{1}{2}}\left( \frac{q}{b}\right) ^{\frac{1}{2}}\e^{-\rho
}\sum_{j=0}^{\infty }\frac{\rho ^{j}}{j!}\int_{0}^{\infty }\exp \left(
-zs\right) ((1+s))^{-\frac{1}{2}-j}(s)^{j-\frac{1}{2}}\mathrm{d}s \\
&=&\left( 2\pi \right) ^{-1}\beta ^{\frac{1}{2}}w_{2}^{-\frac{1}{2}}\left(
\alpha _{1}\beta \sigma _{1}^{2}\right) ^{-\frac{1}{2}}\e^{-\rho
}\sum_{j=0}^{\infty }\frac{\rho ^{j}\Gamma \left( j+\frac{1}{2}\right) }{j!}%
\Psi \left( j+\frac{1}{2},1;z\right) .
\end{eqnarray*}
}

\PP{
By eq. (2) in \cite{Abadir1999} $\Gamma \left( j+\frac{1}{2}\right) =\sqrt{\pi }%
\left( \frac{1}{2}\right) _{j}$, where $(a)_{j}:=\prod_{i=1}^{j-1}(a+i)$
denotes Pochhammer's symbol. Substituting back, noting that $\rho=\rho_w$ and
rearranging, one finds \eqref{eq_predictive_Gaussian_2BIS}, \eqref{eq_predictive_GaussianBIS} and \eqref{eq_cjsstar}.
}

\PP{Next consider the case $h \geq 3$ and
\begin{equation*}
f_{z_{h}|\vvarsigma}(w_{h}|\vs)=\frac{\gamma _{h}^{\frac{1}{2}}w_{h}^{-\frac{%
1}{2}}}{\left( 2\pi \right) ^{\frac{h}{2}}}\int_{\mathbb{R}_{+}^{h-1}}\exp
\left( -\frac{w_{h}}{2\sigma _{h}^{2}}\right) \exp \left( -\frac{1}{2}\left(
\sum_{t=1}^{h-1}\frac{\beta }{\alpha _{t}}v_{t}\right) \right) (\sigma
_{h}^{2})^{-\frac{1}{2}}\prod_{t=1}^{h-1}\frac{\mathrm{d}v_{t}}{\sqrt{v_{t}}}.
\end{equation*}%
Note that $\sigma _{h}^{2}=\omega (1+s)$ with $s:=(1+v_{h-1})\beta \sigma
_{h-1}^{2}/\omega $; next set $\rho :=w_{h}/(2\omega )$ so that $\exp \left(
-w_{h}/(2\sigma _{h}^{2})\right) =\exp \left( -\rho /(1+s)\right) =\exp
(-\rho )\exp \left( \rho s/(1+s)\right) $. }

Setting $p_{j}(\rho ):=(-1)^{j}L_{j}^{(-1)}(\rho )$,
where $%
L_{j}^{(-1)}(\rho )=\sum_{k=0}^{j}\left( k\right) _{j-k}\left( -j\right) _{k}%
\frac{\rho ^{k}}{k!}$ is a generalized Laguerre polynomial, see \cite{Abramowitz1964} formulae 22.2.12,
22.9.16, and
one finds
\begin{equation*}
\exp \left( \frac{s\rho }{1+s}\right) =\sum_{j=0}^{\infty }\left( -s\right)
^{j}L_{j}^{(-1)}(\rho )=\sum_{j=0}^{\infty }p_{j}(\rho
)s^{j}=\sum_{j=0}^{\infty }p_{j}(\rho )(1+v_{h-1})^{j}\beta ^{j}\left(
\sigma _{h-1}^{2}\right) ^{j}\omega ^{-j}.
\end{equation*}%

\PP{Next consider $(\sigma _{h}^{2})^{-\frac{1}{2}}=\omega ^{-\frac{1}{2}%
}(1+s)^{-\frac{1}{2}}$ and expand $(1+s)^{-\frac{1}{2}}$ in decreasing
powers of $s$, which is convergent thanks to Assumption \ref{ass_1}; this implies that
\begin{equation*}
(\sigma _{h}^{2})^{-\frac{1}{2}}=
\omega ^{-\frac{1}{2}}\sum_{k_{1}=0}^{\infty }\binom{-\frac{1}{2}}{k_{1}}%
s^{-\frac{1}{2}-k_{1}}=\sum_{k_{1}=0}^{\infty }\binom{-\frac{1}{2}}{k_{1}}%
\left( 1+v_{h-1}\right) ^{-\frac{1}{2}-k_{1}}\beta ^{-\frac{1}{2}%
-k_{1}}\left( \sigma _{h-1}^{2}\right) ^{-\frac{1}{2}-k_{1}}\omega ^{k_{1}}
\end{equation*}%
Substituting back in $f_{z_{h}|\vvarsigma}(w_{h}|\vs)$ one finds
\begin{align*}
f_{z_{h}|\vvarsigma}(w_{h}|\vs) &=\frac{\gamma _{h}^{\frac{1}{2}}w_{h}^{-%
\frac{1}{2}}}{\left( 2\pi \right) ^{\frac{h}{2}}}\e^{-\rho
}\sum_{j=0}^{\infty }p_{j}(\rho)\sum_{k_{1}=0}^{\infty }\binom{-%
\frac{1}{2}}{k_{1}}\beta ^{j-\frac{1}{2}-k_{1}}\omega ^{k_{1}-j}\cdot  \\
&\cdot \int_{\mathbb{R}_{+}^{h-1}}(1+v_{h-1})^{j-\frac{1}{2}-k_{1}}\left(
\sigma _{h-1}^{2}\right) ^{j-\frac{1}{2}-k_{1}}\exp \left( -\frac{1}{2}%
\left( \sum_{t=1}^{h-1}\frac{\beta }{\alpha _{t}}v_{t}\right) \right)
\prod_{t=1}^{h-1}\frac{\mathrm{d}v_{t}}{\sqrt{v_{t}}}
\end{align*}
Next use the binomial expansion on $\left( \sigma _{h-1}^{2}\right) ^{q}$
with $q=j-\frac{1}{2}-k_{1}$, setting $K_{j}^{\ast }:=\sum_{i=2}^{j}k_{i}$, $%
U_{j}^{\ast }:=\sum_{i=2}^{j}K_{i}^{\ast }$
\begin{align}
f_{z_{h}|\vvarsigma}(w_{h}|\vs) &=\frac{\gamma _{h}^{\frac{1}{2}}w_{h}^{-%
\frac{1}{2}}}{\left( 2\pi \right) ^{\frac{h}{2}}}\e^{-\rho
}\sum_{j=0}^{\infty }p_{j}(\rho )\sum_{k_{1}=0}^{\infty }\binom{-%
\frac{1}{2}}{k_{1}}\beta ^{q}\omega ^{k_{1}-j}\cdot
\nonumber\\
&\cdot \sum_{k_{2}=0}^{q}\sum_{k_{3}=0}^{q-k_{2}}\cdots
\sum_{k_{h-2}=0}^{q-K_{h-3}^{\ast }}\binom{q}{k_{2}}\cdots \binom{%
q-K_{n-3}^{\ast }}{k_{h-2}}\beta ^{(h-3)q-U_{h-2}^{\ast }}\omega
^{K_{h-2}^{\ast }}\cdot
\nonumber\\
&\cdot \int_{\mathbb{R}_{+}^{h-1}}(1+v_{h-1})^{q}\prod_{t=2}^{h-2}\left(
1+v_{h-t}\right) ^{q-K_{t}^{\ast }}\left( 1+(1+v_{1})\beta \frac{\sigma
_{1}^{2}}{\omega }\right) ^{q-K_{n-2}^{\ast }}\cdot
\nonumber\\
&\cdot \exp \left( -\frac{1}{2}\left( \sum_{t=1}^{h-1}\frac{\beta }{\alpha
_{t}}v_{t}\right) \right) \prod_{t=1}^{h-1}\frac{\mathrm{d}v_{t}}{\sqrt{v_{t}%
}}. \label{eq_ff}
\end{align}
Note that $q-K_{t}^{\ast }=j-\frac{1}{2}-K_{t}$ in the previous expression.
The integral in the last two lines of \eqref{eq_ff} can be computed using
Lemma \ref{Lemma_psi_int}, and equals
\begin{equation*}
\pi ^{\frac{h-1}{2}}\left( \frac{\omega +\beta \sigma _{1}^{2}}{\omega }%
\right) ^{j-K_{h-2}}\left( \frac{\beta \sigma _{1}^{2}}{\omega }\right) ^{-%
\frac{1}{2}}\prod_{t=1}^{h-2}\Psi \left( \frac{1}{2},j+1-K_{t};-\frac{\beta
}{2\alpha _{h-t}}\right) \Psi \left( \frac{1}{2},j+1-K_{h-2};\frac{\omega
+\beta \sigma _{1}^{2}}{2\alpha _{1}\sigma _{1}^{2}}\right)
\end{equation*}%
Substituting back and rearranging, one finds \eqref{eq_c_r_h_varsigmaBIS3}.
}
\PP{
Summability of \eqref{eq_predictive_Gaussian_2BIS}, \eqref{eq_predictive_GaussianBIS},
\eqref{eq_cjsstar},
\eqref{eq_c_r_h_varsigmaBIS3}
for any finite $w_h$ or $u_h$ is proved as in the proof of Theorem \ref{theorem_1} using Lemma \ref{lemma-Psi}.
}
\end{proof}

\section{Mapping estimation uncertainty}
\label{sec_grid}
\PP{This Appendix describes how to construct a grid of points in $B_{\eta }$.
The approach is to select points $\vy$ uniformly in $\mathcal{S}$, the
closed unit ball in $r$ dimensions, map them into points $\vtheta$ in $%
A_{\eta }$, and finally apply the exact formulae $\vvarphi$ to obtain points
in $B_{\eta }=\vvarphi(A_{\eta })=\{\vvarphi(\vtheta),\vtheta\in A_{\eta }\}$
in \eqref{eq_CI2}. This creates a grid of points in $B_{\eta }$, over which
one can obtain extremes of the uncertainty region.}

\PP{
Let $\vv=\mB\mR^{\prime }\vtheta$, $\vu=\mB\mR^{\prime }\widehat{\vtheta}$, $%
\mB=(c_{\eta }\mOmega_{\mR})^{-\frac{1}{2}}$ where $\mA^{\frac{1}{2}}$
indicates the symmetric square root of a positive semidefinite matrix $\mA$,
i.e. $\mA^{\frac{1}{2}}=\mU\mLambda^{\frac{1}{2}}\mU^{\prime }$, where $\mA=%
\mU\mLambda\mU^{\prime }$ is the spectral decomposition of $\mA$. The
vectors $\vv$ and $\vu$ are $r\times 1$ vectors, and let $\vx=\vv-\vu$. The
set $A_{\eta }$ in \eqref{eq_CI} corresponds to $C_{\eta }=\{\vv\in \SR%
^{r}:\Vert \vv-\vu\Vert \leq 1\}$, where $\Vert \va\Vert =(\va^{\prime }\va%
)^{\frac{1}{2}}$ is the Euclidean norm. Any point $\vv$ in $C_{\eta }$
corresponds to a unique point $\vx$ in the closed unit ball $\mathcal{S}=\{%
\vx\in \SR^{r}:\Vert \vx\Vert \leq 1\}$ and vice versa. Inverting this map,
any point $\vx$ in $\mathcal{S}$ corresponds to one $\mR^{\prime }\vtheta=\mB%
^{-1}(\vx+\vu)$ in $A_{\eta }$.
}

\PP{
In order to sample points uniformly in $\mathcal{S}$, a simple algorithm is
to draw $\vy$ from a $\rN(0,\mI_{r})$ and $u$ from a $\mathcal{U}_{[0,1]}$, the
uniform distribution on $[0,1]$, with $u$ independent of $\vy$; then $\vx=u^{%
\frac{1}{r}}\vy/\Vert \vy\Vert $ is uniformly distributed in $\mathcal{S}$,
see e.g. \citet{Harman2010}. Finally one can set
\begin{equation}
\mR^{\prime }\vtheta=\mB^{-1}(\vx+\vu)=(c_{\eta }\mOmega_{\mR})^{\frac{1}{2}%
}u^{\frac{1}{r}}\frac{\vy}{\Vert \vy\Vert }+\mR^{\prime }\widehat{\vtheta}
\label{eq_gridA}
\end{equation}%
to find the corresponding point in $A_{\eta }$. Finally apply the exact
formulae $\vvarphi$ to obtain points in $B_{\eta }=\vvarphi(A_{\eta })=\{%
\vvarphi(\vtheta),\vtheta\in A_{\eta }\}$ in \eqref{eq_CI2}.
}

\PP{In the implementation of the calculations behind Table \ref{table_grid}, 100 draws of $\vy$ and $u$ were generated independently. 100 values of $\vx$ were generated as $\vx=u^{\frac{1}{r}}\vy/\Vert \vy\Vert $, obtaining points uniformly distributed within the sphere $\mathcal{S}$. The same 100 draws of $\vy$ were used to generate the corresponding points on the surface of $\mathcal{S}$ by replacing the value of $u$ with 1 in formula \eqref{eq_gridA}. This gave a set of 200 points in $\mathcal{S}$, half of which on the surface.
}

\PP{Because of the asymptotic nature of the confidence ellipsoid, some of the obtained points $\mR^{\prime }\vtheta=\mB%
^{-1}(\vx+\vu)$ contained negative values of $\omega$ or $\alpha$, i.e. were not inside the parameter space. This never happened for the case of Microsoft stock returns, but happened some 20\% of the time in the case of the simulation run; these points $\mR^{\prime }\vtheta$ were discarded.}

\PP{The value of $\sigma_1^2$ was chosen as 3 times $\omega(1+\alpha+\beta)$, because the choice $\omega/(1-\alpha-\beta)$ sometimes gave negative values in the `One simulation run' case, while for the `Microsoft stock return' case  $\sigma_1^2$ was chosen as $\omega/(1-\alpha-\beta)$, because this gave always positive values.}

\end{document}